\documentclass[reqno,11pt]{amsart}
\usepackage[margin=1in]{geometry}

\usepackage{amsthm, amsmath, amssymb, bm}
\usepackage{tikz}

\usepackage[utf8]{inputenc}
\usepackage[T1]{fontenc}

\usepackage[textsize=small,backgroundcolor=orange!20]{todonotes}

\usepackage[hidelinks]{hyperref}
\usepackage{url}

\usepackage[noabbrev,capitalize]{cleveref}
\crefname{equation}{}{}

\usepackage{url}
\newtheorem{theorem}{Theorem}[section]
\newtheorem{proposition}[theorem]{Proposition}
\newtheorem{lemma}[theorem]{Lemma}

\newtheorem{corollary}[theorem]{Corollary}

\theoremstyle{definition}

\theoremstyle{remark}
\newtheorem*{remark}{Remark}

\newcommand{\abs}[1]{\left\lvert#1\right\rvert}
\newcommand{\sabs}[1]{\lvert#1\rvert}

\newcommand{\norm}[1]{\left\lVert#1\right\rVert}
\newcommand{\snorm}[1]{\lVert#1\rVert}

\newcommand{\paren}[1]{\left( #1 \right)}

\newcommand{\set}[1]{\left\{ #1 \right\}}

\newcommand{\wt}{\widetilde}
\newcommand{\wh}{\widehat}

\newcommand{\EE}{\mathbb{E}}
\newcommand{\FF}{\mathbb{F}}
\newcommand{\RR}{\mathbb{R}}
\newcommand{\PP}{\mathbb{P}}

\newcommand{\ZZ}{\mathbb{Z}}

\newcommand{\cP}{\mathcal{P}}
\newcommand{\cQ}{\mathcal{Q}}

\DeclareMathOperator{\codeg}{codeg}
\DeclareMathOperator{\ex}{ex}

\newcommand{\x}{\times}

\title{The regularity method for graphs with few 4-cycles}

\author{David Conlon}
\address{Conlon, Department of Mathematics, California Institute of Technology, Pasadena, CA, USA}
\email{dconlon@caltech.edu}
\thanks{Conlon is supported by NSF Award DMS-2054452 and in part by ERC Starting Grant 676632.}

\author{Jacob Fox}
\address{Fox, Department of Mathematics, Stanford University, Stanford, CA, USA}
\email{jacobfox@stanford.edu}
\thanks{Fox is supported by a Packard Fellowship and by NSF Award DMS-1855635.}

\author{Benny Sudakov}
\address{Sudakov, Department of Mathematics, ETH, Zurich, 8092, Switzerland}
\email{benjamin.sudakov@math.ethz.ch}
\thanks{Sudakov is supported in part by SNSF grant 200021\_196965.}

\author{Yufei Zhao}
\address{Zhao, Department of Mathematics, Massachusetts Institute of Technology, Cambridge, MA, USA}
\email{yufeiz@mit.edu}
\thanks{Zhao is supported by NSF Award DMS-1764176, the MIT Solomon Buchsbaum Fund, and a Sloan Research
Fellowship.}

\begin{document}

\begin{abstract}
	We develop a sparse graph regularity method that applies to graphs with few $4$-cycles,
		including new counting and removal lemmas for 5-cycles in such graphs.
	Some applications include: 
	\begin{itemize}
		\item Every $n$-vertex graph with no $5$-cycle can be made triangle-free by deleting $o(n^{3/2})$ edges.
		\item For $r \geq 3$, every $n$-vertex $r$-graph with girth greater than $5$ has $o(n^{3/2})$ edges.
		\item Every subset of $[n]$ without a nontrivial solution to the equation $x_1 + x_2 + 2x_3 = x_4 + 3x_5$ has size $o(\sqrt{n})$.
	\end{itemize}
\end{abstract}

\maketitle

\section{Introduction}

Szemer\'edi's regularity lemma~\cite{Sze78} is a rough structure theorem that applies to all graphs. The lemma originated in Szemer\'edi's proof of his celebrated theorem that dense sets of integers contain arbitrarily long arithmetic progressions~\cite{Sze75} and is now considered one of the most useful and important results in combinatorics. Among its many applications, one of the earliest was the influential triangle removal lemma of Ruzsa and Szemer\'edi \cite{RSz78}, which says that any $n$-vertex graph with $o(n^3)$ triangles can be made triangle-free by removing $o(n^2)$ edges. Surprisingly, this simple sounding statement is already sufficient to imply Roth's theorem, the special case of Szemer\'edi's theorem for $3$-term arithmetic progressions, and a generalization known as the corners theorem.

Most applications of the regularity lemma, including the triangle removal lemma, rely on also having an associated {\it counting lemma}. Such a lemma roughly says that the number of embeddings of a fixed graph $H$ into a pseudorandom graph $G$ can be estimated by pretending that $G$ is a random graph. This combined application of the regularity lemma and a counting lemma is often referred to as the {\it regularity method} and has had important applications in graph theory, combinatorial geometry, additive combinatorics, and theoretical computer science. For surveys on the regularity method and its applications, we refer the interested reader to \cite{CF13, KS96, RS10}.

The original version of the regularity lemma is only meaningful for dense graphs. However, many interesting and challenging combinatorial problems concern sparse graphs, so it would be extremely valuable to develop a regularity method that also applies to these graphs. The first step in this direction was already taken in the 1990's by Kohayakawa~\cite{Koh97} and R\"odl (see~\cite{GS05}), who proved an analogue of Szemer\'edi's regularity lemma for sparse graphs (see also~\cite{Sco11}). The problem of proving an associated sparse counting lemma has been a more serious challenge, but one that has seen substantial progress in recent years.

It is known that one needs to make some nontrivial assumptions about a sparse graph in order for a counting lemma to hold. Usually, that has meant that the graph is assumed to be a subgraph of another well-behaved sparse graph, such as a random or pseudorandom graph. For subgraphs of random graphs, proving such a counting lemma (or, more accurately in this context, embedding lemma) was a famous problem, known as the K\L R conjecture \cite{KLR97}, which has only been resolved very recently~\cite{BMS15,CGSS14,ST15} as part of the large body of important work (see also~\cite{CG16,Sch16}) extending classical combinatorial theorems such as Tur\'an's theorem and Szemer\'edi's theorem to subsets of random sets.

In the pseudorandom setting, the aim is again to prove analogues of combinatorial theorems, but now for subsets of pseudorandom sets. For instance, the celebrated Green--Tao theorem \cite{GT08}, that the primes contain arbitrarily long arithmetic progressions, may be viewed in these terms. Indeed, the main idea in their work is to prove an analogue of Szemer\'edi's theorem for dense subsets of pseudorandom sets and then to show that the primes form a dense subset of a pseudorandom set of ``almost primes'' so that their relative Szemer\'edi theorem can be applied. For graphs, a sparse counting lemma, which easily enables the transference of combinatorial theorems from dense to sparse graphs, was first developed in full generality by Conlon, Fox, and Zhao~\cite{CFZ14a} and then extended to hypergraphs in~\cite{CFZ15}, where it was used to give a simplified proof of a stronger relative Szemer\'edi theorem, valid under weaker pseudorandomness assumptions than in~\cite{GT08}.

In this paper, we develop the sparse regularity method in another direction, without any assumption that our graph is contained in a sufficiently pseudorandom host. Instead, our only assumption will be that the graph has few $4$-cycles and our main contribution will be a counting lemma that lower bounds the number of $5$-cycles in such graphs.\footnote{Longer cycles, as well as several other families of graphs, can also be counted using our techniques. We hope to return to this point in a future paper.} Unlike the previous results on sparse regularity, the method developed here has natural applications in extremal and additive combinatorics with few hypotheses about the setting. We begin by exploring these applications.

\vspace{2mm}

\textit{Asymptotic notation.} For positive functions $f$ and $g$ of $n$, we write $f = O(g)$ or $f \lesssim g$ to mean that $f \le C g$ for some constant $C > 0$; we write $f = \Omega(g)$ or $f \gtrsim g$ to mean that $f \ge c g$ for some constant $c > 0$; we write $f = o(g)$ to mean that $f/g \to 0$; and we write $f = \Theta(g)$ or $f \asymp g$ to mean that $g \lesssim f \lesssim g$.

\subsection{Sparse graph removal lemmas}

The famous triangle removal lemma of Ruzsa and Szemer\'edi~\cite{RSz78} states that:
\begin{quote}
	\emph{An $n$-vertex graph with $o(n^3)$ triangles 
	can be made triangle-free 
	by deleting $o(n^2)$ edges.
	}
\end{quote}

One of the main applications of our sparse regularity method is a removal lemma for $5$-cycles in $C_4$-free graphs. Since a $C_4$-free graph on $n$ vertices has $O(n^{3/2})$ edges, a removal lemma in such graphs is only meaningful if the conclusion is that we can remove $o(n^{3/2})$ edges to achieve our goal, in this case that the graph should also be $C_5$-free. We show that such a removal lemma holds if our $n$-vertex $C_4$-free graph has $o(n^{5/2})$ $C_5$'s.

\begin{theorem} \label{thm:C4-free-C5-removal}
	An $n$-vertex $C_4$-free graph with $o(n^{5/2})$ $C_5$'s can be made $C_5$-free by removing $o(n^{3/2})$ edges.
\end{theorem}

This theorem is a special case of the following more general result.

\begin{theorem}[Sparse $C_3$--$C_5$ removal lemma]
	 \label{thm:345-removal}
	 An $n$-vertex graph with 
	$o(n^2)$ $C_4$'s
	and
	$o(n^{5/2})$ $C_5$'s
	can be made $\{C_3,C_5\}$-free
	by deleting $o(n^{3/2})$ edges.
\end{theorem}

\begin{remark}
Let us motivate the exponents that appear in this theorem. 
It is helpful to compare the quantities with what is expected in an $n$-vertex random graph with edge density $p$. 
Provided that $pn \to \infty$, the number of $C_k$'s in $G(n,p)$ is typically on the order of $p^k n^k$ for each fixed $k$. Moreover, if $p \gtrsim n^{-1/2}$, then a second-moment calculation shows that $G(n,p)$ typically contains on the order of $p n^2 $ edge-disjoint triangles and so cannot be made triangle-free by removing $o(p n^2)$ edges. Hence, the random graph $G(n,p)$ with $p \asymp n^{-1/2}$ shows that \cref{thm:345-removal} becomes false if we only assume that there are $O(n^{5/2})$ $C_5$'s and $O(n^2)$ $C_4$'s. 

For a different example, we note that the polarity graph of Brown~\cite{Bro66} and Erd\H{o}s--R\'enyi--S\'os~\cite{ERS66} 
has $n$ vertices, $\Theta(n^{3/2})$ edges, no $C_4$'s, 
$\Theta(n^{5/2})$ $C_5$'s, and every edge is contained in exactly one triangle (see~\cite{LV03} for the proof of this latter property). Thus, \cref{thm:345-removal} is false if we relax the hypothesis on the number of $C_5$'s from $o(n^{5/2})$ to $O(n^{5/2})$. For another example, showing that the hypothesis on the number of $C_4$'s also cannot be entirely dropped, we refer the reader to \cref{prop:construction-few-C5} below.
\end{remark}

We state two additional corollaries of \cref{thm:345-removal}, the second being an immediate consequence of the first. See \cref{sec:pf-rem-cor} for the short deductions.

\begin{corollary} \label{cor:o-C5}
	An $n$-vertex graph with 
	$o(n^2)$ $C_5$'s
	can be made triangle-free by deleting $o(n^{3/2})$ edges.
\end{corollary}

\begin{corollary} \label{cor:C5-free}
	An $n$-vertex $C_5$-free graph 
	can be made triangle-free by deleting $o(n^{3/2})$ edges. 
\end{corollary}

We do not know if the exponent $3/2$ in \cref{cor:C5-free} is best possible, but the next statement, whose proof can be found in \cref{sec:constructions}, shows that the hypothesis on the number of $C_5$'s in \cref{cor:o-C5} cannot be relaxed from $o(n^2)$ to $o(n^{5/2})$. 

\begin{proposition} \label{prop:construction-few-C5}
	There exist $n$-vertex graphs with $o(n^{2.442})$ $C_5$'s that cannot be made triangle-free by deleting $o(n^{3/2})$ edges.
\end{proposition}

We also state a $5$-partite version of the sparse $5$-cycle removal lemma. This statement will be used in our arithmetic applications.

\begin{theorem}[Sparse removal lemma for 5-cycles in 5-partite graphs] \label{thm:5-cycle-rem}
	For every $\epsilon > 0$, 
	there exists $\delta > 0 $ such that 
	if $G$ is a $5$-partite graph on vertex sets $V_1, \dots, V_5$ with $|V_1| = \cdots = |V_5| = n$, 
	all edges of $G$ lie between $V_i$ and $V_{i+1}$ for some $i$ (taken mod $5$), 
	and
	\begin{enumerate}
		\item [(a)] (Few $4$-cycles between two parts) $G$ has at most $\delta n^2$ copies of $C_4$ whose vertices lie in two different parts $V_i$,
		\item [(b)] (Few 5-cycles) $G$ has at most $\delta n^{5/2}$ copies of $C_5$, 
	\end{enumerate}
	then $G$ can be made $C_5$-free by removing at most $\epsilon n^{3/2}$ edges.
\end{theorem}

We note that the exponent $3/2$ in the conclusion above is tight, as shown by the next statement, 
    whose proof can be found in \cref{sec:constructions}.

\begin{proposition} \label{prop:C5-construction}
	For every $n$, there exists a 5-partite graph on vertex sets $V_1, \dots, V_5$ with $|V_1| = \cdots = |V_5| = n$, where all edges lie between $V_i$ and $V_{i+1}$ for some $i$ (taken mod $5$), such that the graph is $C_4$-free, every edge lies in exactly one 5-cycle, and there are $e^{-O(\sqrt{\log n})}n^{3/2}$ edges.
\end{proposition}

\emph{Related results.} 
An earlier application of sparse regularity to $C_4$-free (and, more generally, $K_{s,t}$-free) graphs may be found in~\cite{AKSV14}, where it was used to study a conjecture of Erd\H{o}s and Simonovits~\cite{ES82} in extremal graph theory. For instance, they show that if $s = 2$ and $t \geq 2$ or if $s = t = 3$, then the maximum number of edges in an $n$-vertex graph with no copy of $K_{s,t}$ and no copy of $C_k$ for some odd $k \geq 5$ is asymptotically the same as the maximum number of edges in a bipartite $n$-vertex graph with no copy of $K_{s,t}$.

\subsection{Extremal results in hypergraphs}

In an $r$-graph (i.e., an $r$-uniform hypergraph), 
a \emph{$(v,e)$-configuration} is a subgraph with $e$ edges and at most $v$ vertices. 
A central problem in extremal combinatorics is to estimate 
$f_r(n, v, e)$, the maximum number of edges in an $n$-vertex $r$-graph
without a $(v,e)$-configuration. 
For brevity, we drop the subscript when $r=3$, simply writing
$f(n,v,e) := f_3(n,v,e)$.

The systematic study of this function was initiated almost five decades ago by Brown, Erd\H{o}s, and S\'os~\cite{BES73,SEB73}. A famous conjecture that arose from their work~\cite{Erd75, EFR86} asks whether $f(n,e+3,e)=o(n^2)$ for any fixed $e \geq 3$. For $e = 3$, this problem was resolved by Ruzsa and Szemer\'edi~\cite{RSz78}. 
In fact, this $(6,3)$-theorem, rather than the triangle removal lemma, was their original motivation for studying such problems. Their result has been extended in many directions, but the problem of showing that $f(n, e+3, e) = o(n^2)$ remains open for all $e \geq 4$.

Our methods give the following new bound for a problem of this type, which turns out to be equivalent to \cref{cor:C5-free}.

\begin{corollary}  \label{cor:10-5}
$f(n,10,5)=o(n^{3/2})$.	
\end{corollary}

We next explain how to deduce \cref{cor:10-5} from \cref{cor:C5-free}. Suppose $H$ is a $3$-graph on $n$ vertices without a $(10,5)$-configuration. We greedily delete vertices from $H$ 
one at a time if they are in at most four edges. In total, this process deletes at most $4n$ edges. The resulting induced subgraph $H'$ has the property that each vertex is in at least $5$ edges. Furthermore, 
$H'$ is linear, that is, any two edges intersect in at most one vertex. Indeed, if there are two edges sharing vertices $u, v$, then, by adding three additional edges touching $v$, we get a $(10,5)$-configuration. 
If we now let $G$ be the underlying graph formed by converting all edges of $H'$ to triangles, we see that $G$ is a union of edge-disjoint triangles. Moreover, it is $C_5$-free, since otherwise it would contain a $(10,5)$-configuration. Hence, by \cref{cor:C5-free}, it has $o(n^{3/2})$ edges. But this then implies that $H'$ and, therefore, $H$ has $o(n^{3/2})$ edges.
Conversely, to show that \cref{cor:10-5} implies \cref{cor:C5-free}, 
	it suffices to observe that the 3-graph formed by a collection of edge-disjoint triangles in a $C_5$-free graph 
	does not have a $(10,5)$-configuration.
	
\medskip

A \emph{(Berge) cycle} of length $k \ge 2$ (or simply a \emph{$k$-cycle}) 
in a hypergraph is an alternating sequence of distinct vertices and edges $v_1,e_1,\ldots,v_k,e_k$
such that $v_i,v_{i+1} \in e_i$ for each $i$ (where indices are taken modulo $k$). 
For example, a 2-cycle consists of a pair of edges intersecting in a pair of distinct vertices.
The \emph{girth} of an $r$-graph is the length of the shortest cycle. 

Let $h_r(n,g)$ denote the maximum number of edges in 
	an $r$-graph on $n$ vertices of girth larger than $g$.
The following observation, whose proof may be found in \cref{sec:pf-girth}, relates the Brown--Erd\H{o}s--S\'os problem to that of estimating $h_r(n, g)$.

\begin{proposition} \label{prop:girth}
	For $r \ge 2$ and $e \ge 2$, there exists $n_0(r,e)$ such that $f_r(n,(r-1)e,e) = h_r(n,e)$ for all $n \ge n_0(r,e)$.
\end{proposition}

By \cref{cor:10-5}, we thus have the following result for $r=3$.  Note that the general result follows from the case $r=3$. Indeed, if an $r$-graph has girth $g$, replacing each edge by a subset of size three, we get a $3$-graph on the same set of vertices with the same number of edges and girth at least $g$. 

\begin{corollary} \label{cor:girth>5}
	Let $r \ge 3$. Then $h_r(n,5) = o(n^{3/2})$, i.e., every $r$-graph on $n$ vertices of girth greater than 5 has $o(n^{3/2})$ edges. 
\end{corollary}

\emph{Related results.} 
Previously, upper bounds of the form $h_r(n,5) \le c_r n^{3/2}$ were known~\cite{LV03,EM18ar}.
In fact, Lazebnik and Verstra\"ete~\cite{LV03} 
showed that $f(n,8,4) = h_3(n,4) = (1/6 + o(1))n^{3/2}$ for all sufficiently large $n$.\footnote{Lazebnik and Verstraete~\cite{LV03} actually claim that $f(n,8,4) = h_3(n,4)$ for all $n$. However, this is false for small $n$. For instance, it is easy to see that $f(6, 8, 4) = 3$, while $h_3(6,4) = 2$. Nevertheless, their claim that $f(n, 8, 4) = (1/6 +o(1))n^{3/2}$ still stands by our \cref{prop:girth} and their result that $h_3(n, 4) = (1/6 +o(1))n^{3/2}$.}
Bollob\'as and Gy\H{o}ri \cite{BG08} proved that the maximum number of edges in a $3$-graph on $n$ vertices with no 5-cycle is $\Theta(n^{3/2})$, which implies that the maximum number of triangles in an $n$-vertex $C_5$-free graph is $\Theta(n^{3/2})$. In contrast, \cref{cor:C5-free} says that the maximum number of \emph{edge-disjoint} triangles in a $C_5$-free graph on $n$ vertices is $o(n^{3/2})$. See \cite{AS16,EM18ar,EMSG19,FO17} for further improvements and simplifications of the Bollob\'as--Gy\H{o}ri result.

\medskip

Given a family $\mathcal{F}$ of $r$-graphs,
	we say that an $r$-graph is \emph{$\mathcal F$-free} if it contains no copy of any element of $\mathcal F$ as a subgraph. Define $\ex(n,\mathcal{F})$ to be the maximum number of edges in an $\mathcal F$-free $r$-graph on $n$ vertices.
It is easy to see that
\[
2^{\ex(n,\mathcal{F})} 
\le 
\#\{\text{$\mathcal F$-free $r$-graphs on $n$ labeled vertices}\} 
\le 
\sum_{m=0}^{\ex(n,\mathcal{F})} \binom{\binom{n}{r}}{m} \le 2^{\ex(n,\mathcal{F})r\log n}.
\]
In many instances it is known that the lower bound is closer to the truth. Early invesigations into this and related problems led to the first seeds of the important container method, as developed by Kleitman and Winston \cite{KW82} 
and by Sapozhenkho \cite{Sap01,Sap02,Sap03}.
More recently, influential works of  
Balogh, Morris, and Samotij \cite{BMS15} 
and 
of Saxton and Thomason \cite{ST15}
pushed these ideas considerably further and showed their broad applicability.

Using the container method, 
Palmer, Tait, Timmons, and Wagner~\cite{PTTW19}
proved that the number of $r$-graphs on $n$ vertices with girth greater than 4 is at most $2^{c_rn^{3/2}}$ for an appropriate constant $c_r$. 
We improve this bound when the girth is greater than 5, strengthening Corollary \ref{cor:girth>5}. 
We refer the reader to \cref{sec:pf-count-girth>5} for the proof.

\begin{theorem} \label{thm:count-girth>5}
For every fixed $r \ge 3$,
the number of $r$-graphs 
on $n$ vertices 
with girth greater than $5$
is $2^{o(n^{3/2})}$.
\end{theorem}

\subsection{Number-theoretic applications} 
It was already noted by Ruzsa and Szemer\'edi that their results imply Roth's theorem~\cite{Rot53}, the statement that every subset of $[n] := \{1, 2, \dots, n\}$ without 3-term arithmetic progressions has size $o(n)$. Here we discuss some number-theoretic applications of our sparse removal results along similar lines. We first illustrate our results with two specific applications, beginning with the following theorem.\footnote{Though we focus here on applications to linear equations with five variables, all of our results extend to linear equations with more than five variables by using the counting lemma for longer cycles mentioned in an earlier footnote.}

\begin{theorem} \label{thm:avoid-eqn1}
	Every subset of $[n]$ without a nontrivial solution to the equation
	\begin{equation}\label{eq:avoid-eqn1}
		x_1 + x_2 + 2x_3 = x_4 + 3x_5
	\end{equation}
	has size $o(\sqrt{n})$. Here a trivial solution is one of the form $(x_1, \dots, x_5) = (x,y,y,x,y)$ or $(y,x,y,x,y)$ for some $x,y \in \ZZ$.
\end{theorem}

Any set of integers without a nontrivial solution to \cref{eq:avoid-eqn1} must be a \emph{Sidon set}, with no nontrivial solution to the equation $x_1 + x_2 = x_4 + x_5$, since any nontrivial solution automatically extends to a nontrivial solution of \cref{eq:avoid-eqn1} by setting $x_3 = x_5$.
In particular, the upper bound for the size of Sidon sets, $(1+o(1))\sqrt{n}$, is also an upper bound for the size of a subset of $[n]$ without a nontrivial solution to \cref{eq:avoid-eqn1}. Our \cref{thm:avoid-eqn1} improves on this simple bound, though it remains an open problem to determine whether the bound can be improved further to $n^{1/2 - \epsilon}$ for some $\epsilon > 0$.

We now give a second number-theoretic application, this time restricting to Sidon sets.

\begin{theorem} \label{thm:avoid-two-eqn1}
	The maximum size of a Sidon subset of $[n]$ without a solution in distinct variables to the equation
	\[
	x_1 + x_2 + x_3 + x_4 = 4x_5
	\]
	is at most $o(\sqrt{n})$ and at least $n^{1/2 - o(1)}$.
\end{theorem}

In other words, we are simultaneously avoiding
\begin{itemize}
	\item[(a)] nontrivial solutions to the Sidon equation $x_1 + x_2 = x_3 + x_4$ and
	\item[(b)] distinct variable solutions to the linear equation $x_1 + x_2 + x_3 + x_4 = 4 x_5$.
\end{itemize}
There exist Sidon sets of size $(1+o(1))\sqrt{n}$, as well as
sets of size $n^{1-o(1)}$ avoiding (b) 
(by a standard modification of Behrend's construction~\cite{Beh46} of large sets without 3-term arithmetic progressions). 
However, \cref{thm:avoid-two-eqn1} shows that 
by simultaneously avoiding nontrivial solutions to both equations, the maximum size is substantially reduced. 

This is the first example showing a lack of ``compactness'' for linear equations. In extremal graph theory, 
	the \emph{Erd\H{o}s--Simonovits compactness conjecture \cite{ES82}} is a well-known conjecture saying that, 
	for every finite set $\mathcal F$ of graphs, $\ex(n, \mathcal F) \ge c_{\mathcal F} \min_{F \in \mathcal F} \ex(n, F)$ for some constant $c_{\mathcal F} > 0$. 
The analogous statement is false for $r$-graphs with $r \ge 3$ (by the Ruzsa--Szemer\'edi theorem and a simple generalisation to $r$-graphs noted in \cite[Theorem 1.9]{EFR86}), but remains open for graphs. Our \cref{thm:avoid-two-eqn1} shows that it also fails for linear equations. 

\cref{thm:avoid-two-eqn1} also sheds some light on the fascinating open problem (see, for example, Gowers' blog post~\cite{Gow12bl}) of understanding the structure of Sidon sets with near-maximum size, say within a constant factor of $\sqrt{n}$, showing that any such set must contain five distinct elements with one of them being the average of the others. More generally, we have the following result, showing that a large Sidon set must contain solutions to a wide family of translation-invariant linear equations in five variables. We note that the lower bound of $n^{1/2-o(1)}$ simply comes from intersecting a Sidon set of set $(1+o(1))n^{1/2}$ with a random translate of a subset of $[n]$ of size $n^{1 - o(1)}$ that avoids nontrivial solutions to \cref{eq:4avg} (which again exists by a standard modification of Behrend's construction~\cite{Beh46}).

\begin{theorem}
	\label{thm:avoid-two-eqn2} 
	Fix positive integers $a_1, \dots, a_4$.
	The maximum size of a Sidon subset of $[n]$ without a solution in distinct variables to the equation
	\begin{equation}
		\label{eq:4avg}
		a_1x_1+a_2x_2+a_3x_3+a_4x_4=(a_1+a_2+a_3+a_4)x_5
	\end{equation}
	is at most $o(\sqrt{n})$ and at least $n^{1/2 - o(1)}$. 
\end{theorem}

Similarly, \cref{thm:avoid-eqn1} is a special case of the following statement.

\begin{theorem}
	\label{thm:avoid-eqn2}
	Fix positive integers $a$ and $b$.
	Every subset of $[n]$ without a nontrivial solution to the equation
	\begin{equation}\label{eq:avoid-eqn2}
		a x_1 + a x_2 + b x_3 = a x_4 + (a+b) x_5
	\end{equation}
	has size $o(\sqrt{n})$. Here a trivial solution is one of the form $(x_1, \dots, x_5) = (x,y,y,x,y)$ or $(y,x,y,x,y)$ (or $(y, y, x, x, y)$ if $a = b$) for some $x,y \in \ZZ$.
\end{theorem}

Both \cref{thm:avoid-eqn2} and the upper bound in \cref{thm:avoid-two-eqn2} are special cases of the following more robust theorem (applied with $X_1 = \cdots = X_5$), whose proof can be found in \cref{sec:pf-nt}.
Indeed, to prove the upper bound in \cref{thm:avoid-two-eqn2}, we apply \cref{lem:non-distinct} below to check that every Sidon subset of $[n]$ contains $O(n)$ solutions to \cref{eq:4avg} where not all variables are distinct, thereby verifying hypothesis (b) of \cref{thm:transl-invar} (with a $O(n)$ bound instead of $o(n^{3/2})$). To prove \cref{thm:avoid-eqn2}, we note, by setting $x_3 = x_5$ in \cref{eq:avoid-eqn2}, that the subset satisfying the hypothesis of \cref{thm:avoid-eqn2} must be a Sidon set. Finally, when $X_1 = \cdots = X_5$, the removal statement in the conclusion of \cref{thm:transl-invar} implies that $\abs{X_1} = o(\sqrt{n})$ since $x_1 = \cdots = x_5$ is always a solution due to $a_1 + \cdots + a_5 = 0$.

\begin{theorem}
	\label{thm:transl-invar} 
	Fix nonzero integers $a_1, \dots, a_5$ with $a_1 + \cdots + a_5 = 0$. 
	Suppose that $X_1, \dots, X_5$ are subsets of $[n]$ satisfying
	\begin{itemize}
		\item[(a)] each $X_i$ has $o(n)$ nontrivial solutions to $x_1 + x_2 = x_3 + x_4$ \\ (a trivial solution here is one with $(x_1, x_2) = (x_3, x_4)$ or $(x_4,x_3)$) and 
		\item[(b)] $o(n^{3/2})$ solutions to $a_1 x_1 + \cdots + a_5 x_5 = 0$ with $x_1 \in X_1$, \dots, $x_5 \in X_5$.
	\end{itemize}
	Then one can remove $o(\sqrt{n})$ elements from each $X_i$ to remove all solutions to $a_1x_1+\cdots+a_5x_5=0$ with $x_1 \in X_1$, \dots, $x_5 \in X_5$.
\end{theorem}

\begin{lemma} \label{lem:non-distinct}
Let $a_1,\dots,a_4$ be nonzero integers and $X$ be a Sidon subset of $[n]$. 
Then $X$ contains $O(n)$ solutions to the equation $a_1 x_1 + a_2 x_2 + a_3 x_3  + a_4 x_4 = 0$.
\end{lemma}

Since our proofs rely on the graph removal lemma, they give poor quantitative bounds, the best bound on that lemma~\cite{Fox11} having tower-type dependencies. However, for our number-theoretic applications, it is possible to use the best bounds for the relevant Roth-type theorem, together with a weak arithmetic regularity lemma and our $C_5$-counting lemma to obtain reasonable bounds. This is similar in spirit to the arithmetic transference proof of the relative Szemer\'edi theorem given in~\cite{Zha14}, though we omit the details. 
A follow-up work of Prendiville~\cite{Pre} giving a Fourier-analytic proof of our number-theoretic results also yields comparable bounds.

As a final remark, we note that the results of this subsection carry over essentially verbatim to arbitrary abelian groups. Following Kr\'al'--Serra--Vena~\cite{KSV09} (see also~\cite{KSV12,Sha10}), one may also use our sparse graph removal lemma to derive a sparse removal lemma that is meaningful in arbitrary groups. We again omit the details, but refer the interested reader to~\cite[Theorem~1.2]{CFZ14} for a result which is similar in flavor.

\section{A weak sparse regularity lemma} \label{sec:regularity}

In this section, we develop a sparse version of the Frieze--Kannan weak regularity lemma~\cite{FK99}.
A sparse version of Szemer\'edi's regularity lemma was originally developed by Kohayakawa \cite{Koh97} and R\"odl (see \cite{GS05}) under an additional ``no dense spots'' hypothesis, but Scott~\cite{Sco11} showed that, with a slight variation in the statement, this additional hypothesis is not needed. The approach we use here for proving a sparse version of the weak regularity lemma will be similar to that of Scott.
With this result (and an appropriate counting lemma) in hand, we will then be able to ``transfer'' the removal lemma from the dense setting to the sparse setting.

In order to give an analytic formulation of the weak regularity lemma, we first make some definitions.
Given a pair of probability spaces $V_1$ and $V_2$, 
 	which are usually vertex sets with the uniform measure 
 	(or, if the vertices carry weights, then with the probability measure that is proportional to the vertex weights), 
 	we define the \emph{cut norm} for a measurable function $f \colon V_1 \times V_2 \to \RR$ 
 	(we will sometimes omit mentioning the measurability requirement when it is clear from context) 
 	by 
\begin{equation} \label{eq:cut-set}
\norm{f}_\square := 
\sup_{\substack{A \subset V_1 \\ B \subset V_2}} \abs{\EE_{x \in V_1, y \in V_2} f(x,y)1_A(x) 1_B(y)},
\end{equation}
where $A$ and $B$ range over all measurable subsets and $x$ and $y$ are chosen independently according to the corresponding probability measures.

Given a partition $\cP$ of some probability space $V$ and a function $f \colon V \times V \to \RR$, we write $f_\cP \colon V \times V \to \RR$ for the function obtained from $f$ by ``averaging'' over blocks $A \times B$ where $A$ and $B$ are parts of $\cP$, i.e., $f_{\cP}(x,y) = \frac{1}{\mu(A)\mu(B)} \int_{A \times B} f$ for all $(x,y) \in A\times B$, where $\mu( \cdot )$ is the probability measure on $V$. We may ignore zero-measure parts.

The weak regularity lemma of Frieze and Kannan may be rephrased in the following way (for example, see \cite[Corollary 9.13]{Lov12}), saying that all bounded functions can be approximated in terms of the cut norm by a step function with a bounded number of blocks. Furthermore, the step function can be obtained by averaging the original function over steps. This analytic perspective on the weak regularity lemma has been popularized by the development of graph limits~\cite{LS07}.

\begin{theorem}[Weak regularity lemma, dense setting] 
Let $\epsilon > 0$. Let $V$ be a probability space and $f \colon V \times V \to [0,1]$ be a measurable symmetric function (i.e., $f(x,y) = f(y,x)$ for all $x,y \in V$). Then there exists a partition $\cP$ of $V$ into at most $2^{O(\epsilon^{-2})}$ parts such that
\[
	\norm{f - f_\cP}_\square \le \epsilon.
\]
\end{theorem}

For sparse graphs, one would like to have control on the error term that is commensurate with the overall edge density of the graph. More explicitly,
for an $n$-vertex graph with on the order of $pn^2$ edges for some $p = p_n \to 0$, one would like to have error terms of the form $\epsilon p$ in the above inequality. 
To capture this scaling, we renormalize $f$ by dividing the edge-indicator function of the graph by the edge density $p$. 
Thus, the sparse setting corresponds to unbounded functions $f$ with $L^1$ norm $O(1)$.

Our sparse weak regularity lemma is stated below. The proof follows an energy increment strategy, as is usual with proofs of regularity lemmas. Since this is now fairly standard, we refer the reader to \cref{sec:pf-reg} for the details. 

\begin{theorem} \label{thm:reg-fn}
Let $\epsilon > 0$. Let $V$ be a probability space and $f \colon V \times V \to [0,\infty)$ be a measurable symmetric function.
Then there exists a partition $\cP$ of $V$ into at most $2^{32 \EE f/\epsilon^2}$ parts such that 
\[
\norm{ (f - f_\cP) 1_{f_\cP \le 1}}_\square \le \epsilon.
\]
\end{theorem}

Here $(f - f_\cP) 1_{f_\cP \le 1}$ denotes the function with value $f(x,y) - f_\cP(x,y)$ if $f_\cP(x,y) \le 1$ and $0$ otherwise.
The cutoff term $1_{f_\cP\le 1}$ means that we neglect those parts of the graph that are too dense.
Indeed, it would be too much to ask for $\norm{f - f_\cP}_\square$ to be small without this cutoff. 
For example, if one had $|V| = n$, $f(x,x) = n$, and $f(x,y) = 0$ for all $x \ne y$, then one would not be able to partition $V$ into $O_\epsilon(1)$ parts so that $\norm{f - f_\cP}_\square \le \epsilon$.

For our applications, it will be necessary to show that the cutoff has little effect on graphs with few $C_4$'s. In the next two lemmas, we show that such graphs have a negligible number of edges lying between pairs of parts whose edge density greatly exceeds the average.
A similar argument was also presented in \cite{AKSV14}, though we include the complete proof here for the convenience of the reader.

\begin{lemma} \label{lem:c4-lower}
	Let $G$ be a bipartite graph with nonempty vertex sets $A$ and $B$ and $m$ edges. 
	If $m \ge 4 \abs{A}\abs{B}^{1/2} + 4\abs{B}$, 
		then the number of $4$-cycles in $G$ is at least $m^4 \abs{A}^{-2} \abs{B}^{-2}/64$.
\end{lemma}

\begin{proof}
	Writing $\binom{x}{2} = x(x-1)/2$ for all real $x$, we have $\binom{x}{2} \ge x^2/4$ for all $x \ge 2$ and so, by the convexity of $x \mapsto \binom{x}{2}$,
	\[
	\sum_{\{x,y\} \in \binom{A}{2}} \codeg(x,y) 
	= \sum_{v \in B} \binom{\deg(v)}{2}
	\ge \abs{B} \binom{m/\abs{B}}{2} \ge \frac{m^2}{4\abs{B}},
	\]
	where $\codeg(x,y)$ is the number of common neighbors of $x$ and $y$. 
	Thus, the number of $4$-cycles in $G$ is
	\[
	\sum_{\{x,y\} \in \binom{A}{2}} \binom{\codeg(x,y)}{2}
	\ge \binom{\abs{A}}{2} \binom{\binom{\abs{A}}{2}^{-1} \frac{m^2}{4\abs{B}}}{2} 
	\ge \frac{m^4}{64 \abs{A}^2\abs{B}^2}. \qedhere
	\]
\end{proof}

We write $e(U, W)$ for the number of pairs $(u,w)\in U \times W$ forming an edge in the given graph $G$.

\begin{lemma}[Dense pairs] \label{lem:c4-dense-pairs}
Let $G$ be an $n$-vertex graph and $V(G) = V_1 \cup V_2 \cup \cdots \cup V_M$ be a partition of the vertex set of $G$. Let $q > 0$ be a real number. Let $T$ denote the number of $4$-cycles in $G$. Then the number of edges of $G$ lying between parts $V_i$ and $V_j$ with $e(V_i, V_j) \ge q\abs{V_i}\abs{V_j}$ (here $i=j$ is allowed) is $O( (1/q + M^2 + M^{1/4} T^{1/4})n)$.
\end{lemma}

\begin{proof}
	For each $i \in [M]$, let $W_i$ denote the union of all $V_j$ such that $e(V_i, V_j) \ge q\abs{V_i}\abs{V_j}$. Then $e(V_i, W_i) \ge q\abs{V_i}\abs{W_i}$. Let $ T_i$ denote the number of $4$-cycles with the first and third vertices in $V_i$ and second and fourth vertices in $W_i$. 
	We claim that
	\begin{equation} \label{eq:c4-dense-pair}
		e(V_i, W_i) \lesssim  \frac{n}{Mq} + \frac{\abs{V_i}}{q} +  Mn + n^{1/2} \abs{V_i}^{1/2}  T_i^{1/4}.
	\end{equation}
	Indeed, if $\abs{V_i} \le 1/Mq$, then $e(V_i, W_i) \le \abs{V_i}\abs{W_i} \le n/Mq$. 
	So assume $\abs{V_i} > 1/Mq$. 
	If $e(V_i, W_i) < 4 \abs{V_i}\abs{W_i}^{1/2} + 4\abs{W_i}$, then $e(V_i, W_i)^2 \leq 32\abs{V_i}^2\abs{W_i}+32\abs{W_i}^2$ and
	\[
		e(V_i, W_i) \le \frac{e(V_i, W_i)^2}{q \abs{V_i}\abs{W_i}}
	\lesssim  \frac{\abs{V_i}}{q} + \frac{\abs{W_i}}{q\abs{V_i}}
	\le  \frac{\abs{V_i}}{q} + Mn,
	\]
	where in the final inequality we use $\abs{W_i} \le n$ and $\abs{V_i} > 1/Mq$.
	Otherwise, by \cref{lem:c4-lower}, we have $T_i \gtrsim e(V_i, W_i)^4 \abs{V_i}^{-2} \abs{W_i}^{-2}$, so
	\[
		e(V_i, W_i) \lesssim \abs{V_i}^{1/2} \abs{W_i}^{1/2}  T_i^{1/4}
		\le n^{1/2} \abs{V_i}^{1/2}  T_i^{1/4}.
	\]
	This proves \cref{eq:c4-dense-pair}. 
	Summing over all $i$, we obtain that the number of edges of $G$ lying between parts $V_i$ and $V_j$ with $e(V_i, V_j) \ge q\abs{V_i}\abs{V_j}$ is at most 
	\begin{align*}
	\sum_{i=1}^M e(V_i, W_i) 
	&\lesssim \sum_{i=1}^M \paren{ \frac{n}{Mq} + \frac{\abs{V_i}}{q} + Mn + n^{1/2} \abs{V_i}^{1/2}  T_i^{1/4}}
	\\
	& \le \frac{2n}{q} + M^2n + n^{1/2} \paren{\sum_{i=1}^M 1}^{1/4} \paren{\sum_{i=1}^M \abs{V_i}}^{1/2} \paren{\sum_{i=1}^M  T_i}^{1/4}
	\\
	&\le \frac{2n}{q} + M^2n + 2 M^{1/4} T^{1/4} n ,  
	\end{align*}
	where we applied H\"older's inequality in the second step and also that $\sum_i T_i \leq 2T$.
\end{proof}

\noindent
\emph{Remark.} We will apply this lemma with $q = K/\sqrt{n}$ for some large constant $K$ and $T=o(n^2)$.

\section{A sparse counting lemma for 5-cycles} \label{sec:counting}

This section contains a novel counting lemma for 5-cycles in sparse graphs. As discussed in the introduction, counting lemmas are often the key obstacles to the development of the sparse regularity method. As such, our counting result, which says that if a graph does not have too many copies of $C_4$, then the number of copies of $C_5$ is approximated by the $C_5$ count in the weak regularity approximation, may be seen as our main contribution.
	
Combinatorially, the intuition is that in a graph with few $C_4$'s and edge density on the order of $n^{-1/2}$, the second neighborhood of a typical vertex has linear size, so that regularity ensures enough edges within this second neighborhood to generate many 5-cycles. This observation was already used to considerable effect in~\cite{AKSV14}.

Let $V$ be a probability space. Given a symmetric function $f\colon V \times V \to [0,\infty)$ (again \emph{symmetric} means $f(x,y) = f(y,x)$),
the \emph{homomorphism density} of $H$ in $f$ is given by
\[
t(H, f) = \int_{V^{V(H)}} \prod_{uv \in E(H)} f(x_u, x_v) \, \prod_{v \in V(H)} dx_v.
\]
As in our formulation of the regularity lemma in \cref{sec:regularity}, one should think of the function $f$ as a normalized edge-indicator function of the form $p^{-1}1_{E(G)}$. Our counting lemma is now as follows.

\begin{theorem}
Let $0< \epsilon < 1$ and $C \geq 1$. Let $V$ be a probability space and let $f \colon V \times V \to [0,\infty)$ and $g \colon V \times V \to [0,1]$ be measurable symmetric functions satisfying $t(C_4, f) \le C$ and $\norm{f-g}_\square \le \epsilon^4$. Then
\[
t(C_5, f) \ge t(C_5, g) - 11 C\epsilon.
\]
\end{theorem}

In practice, we will prove a multipartite version of this counting lemma which will be useful for applications. The two versions are essentially equivalent.

\vspace{3mm}

\emph{Notation and conventions.}
Let $V_1, \dots, V_5$ be probability spaces with indices considered mod $5$.
Vertices in $V_i$ are denoted by $x_i$ and, when appearing in the subscript of an expectation symbol (e.g., $\EE_{x_i}$), it is assumed that $x_i$ varies independently over $V_i$ according to its probability distribution (e.g., a uniformly random vertex if $V_i$ comes from an unweighted graph). 
For all $i\ne j$ differing by one in $\ZZ/5\ZZ$, we write $f_{ij} = f_{i,j}$ for a measurable function on $V_i \times V_j$ and assume that $f_{ij}(x_i, x_j) = f_{ji}(x_j, x_i)$.
Given functions $f_{12} \colon V_1 \times V_2 \to \RR$ and $f_{23} \colon V_2 \times V_3 \to \RR$, we write $f_{12} \circ f_{23} \colon V_1 \times V_3 \to \RR$ for the function 
\[
(f_{12} \circ f_{23}) (x_1, x_3) = \EE_{x_2 \in V_2} f_{12}(x_1, x_2) f_{23}(x_2, x_3),
\]
which can be thought of as the number of 2-edge paths between the vertices $x_1$ and $x_3$, appropriately normalized. The notation $\circ$ is used since it can be viewed as a composition of linear operators.

\begin{theorem} \label{thm:sparse-c5-count}
	Let $0 < \epsilon < 1$ and $C \ge 1$.
	Let $V_1, \dots, V_5$ be probability spaces. For each $i$ (taken mod $5$), let $f_{i,i+1},g_{i,i+1} \colon V_i \times V_{i+1} \to [0,\infty)$ with $0 \le g_{i,i+1} \le 1$ pointwise,  $\norm{f_{i,i+1} - g_{i,i+1}}_{\square} \le \epsilon^4$, and
	\begin{equation}
		\label{eq:sparse-c4-cond}
		\norm{f_{i-1,i} \circ f_{i,i+1}}_2^2 \le C.
	\end{equation}
	Then
	\begin{equation}
		\label{eq:sparse-c5-count}
	 \EE_{x_1, \dots, x_5} \prod_{i=1}^5 f_{i,i+1}(x_i,x_{i+1}) \ge \EE_{x_1, \dots, x_5} \prod_{i=1}^5 g_{i,i+1}(x_i,x_{i+1}) - 11 C \epsilon.
	\end{equation}
\end{theorem}

\begin{remark}
For an $n$-vertex graph with edge density on the order of $p$, 
the hypothesis \cref{eq:sparse-c4-cond} translates into a $O(n^4p^4)$ upper bound on the number of $4$-cycles with vertices in $V_{i-1}, V_i, V_{i+1}, V_i$ in turn, since
\[
\norm{f_{01} \circ f_{12}}_2^2 
= \EE_{x_0,x_1,x'_1,x_2} f_{01}(x_0,x_1) f_{01}(x_0,x'_1) f_{12}(x_1,x_2) f_{12}(x'_1,x_2).
\]	
Applying the Cauchy--Schwarz inequality, we have
\begin{align*}
\norm{f_{01} \circ f_{12}}_2^2 
&= \EE_{x_1,x'_1} [(\EE_{x_0} f_{01}(x_0,x_1) f_{01}(x_0,x'_1)) (\EE_{x_2} f_{12}(x_1,x_2) f_{12}(x'_1,x_2))]
\\
&\le \left(\EE_{x_1,x'_1} [(\EE_{x_0} f_{01}(x_0,x_1) f_{01}(x_0,x'_1))^2]\right)^{1/2} \left(\EE_{x_1,x'_1} [(\EE_{x_2} f_{12}(x_1,x_2) f_{12}(x'_1,x_2))^2]\right)^{1/2}
\\
&= \left(\EE_{x_0, x'_0, x_1,x'_1} [f_{01}(x_0,x_1)f_{01}(x'_0,x_1) f_{01}(x_0,x'_1) f_{01}(x'_0,x'_1)]\right)^{1/2} \\
&\qquad \cdot 
\left(\EE_{x_1, x'_1, x_2,x'_2} [f_{12}(x_1,x_2)f_{12}(x'_1,x_2) f_{12}(x_1,x'_2) f_{12}(x'_1,x'_2)]\right)^{1/2}
\\
&= t(K_{2,2}, f_{01})^{1/2} t(K_{2,2}, f_{12})^{1/2}.
\end{align*}
So we could replace \cref{eq:sparse-c4-cond} by the stronger condition that $t(K_{2,2}, f_{i,i+1}) \le C$ for each $i$, i.e., a $O(n^4 p^4)$ bound on the number of $4$-cycles in the bipartite graph between $V_i$ and $V_{i+1}$ for every $i$.
\end{remark}

Let us collect some basic facts about the cut norm. It is a standard fact that the definition of the cut norm~\cref{eq:cut-set} is equivalent to
\begin{equation} \label{eq:cut-fn}
\norm{f}_\square = 
\sup_{\substack{a \colon V_1 \to [0,1]  \\ b \colon V_2 \to [0,1]}} \abs{\EE_{x \in V_1, y \in V_2} f(x,y) a(x) b(y)},
\end{equation}
where now $a$ and $b$ are measurable functions taking values in $[0,1]$. The equivalence can be seen since the expectation expression in \cref{eq:cut-fn} is bilinear in $a$ and $b$ and thus its extrema must occur when $a$ and $b$ take values from the set $\{0,1\}$. But then $a$ and $b$ may be thought of as indicator functions of sets $A$ and $B$, returning us to the earlier definition~\cref{eq:cut-set} of the cut norm.

\begin{lemma} \label{lem:compose-cut}
	Let $f_{12} \colon V_1 \times V_2 \to \RR$ and $f_{23} \colon V_2 \times V_3 \to [0,\infty)$ with $\EE_{x_3 \in V_3} f_{23}(x_2, x_3) \le M$ for all $x_2 \in V_2$. Then
	\[
	\norm{f_{12} \circ f_{23}}_\square \le  M \norm{f_{12}}_\square.
	\]
\end{lemma}

\begin{proof}
	For any $u_1 \colon V_1 \to [0,1]$ and $u_3 \colon V_3 \to [0,1]$, one has
	\begin{align*}
	\EE_{x_1, x_3} u_1(x_1) (f_{12} \circ f_{23})(x_1, x_3) u_3(x_3) &= \EE_{x_1, x_2, x_3} u_1(x_1) f_{12}(x_1, x_2) f_{23}(x_2, x_3) u_3(x_3)
	\\
	&= M \cdot \EE_{x_1, x_2} u_1(x_1) f_{12}(x_1, x_2) u_2(x_2),
	\end{align*}
	where 
	\[
	u_2(x_2) =\frac{1}{M} \EE_{x_3} f_{23}(x_2, x_3) u_3(x_3) \leq \frac{1}{M} \EE_{x_3} f_{23}(x_2, x_3) \leq 1. 
	\]
	The claim then follows from \cref{eq:cut-fn} since $u_1(x_1) \in [0,1]$ and $u_2(x_2) \in [0,1]$.
\end{proof}

\begin{lemma}[Triangle counting lemma for dense graphs] \label{lem:dense-triangle}
	Let $f_{ij} \colon V_i \times V_j \to \RR$ for $ij \in \{12,13,23\}$ and assume that $f_{12}$ and $f_{13}$ take only nonnegative values. Then
	\[
	\EE_{x_1 \in V_1,x_2 \in V_2, x_3 \in V_3} f_{12}(x_1,x_2) f_{13}(x_1,x_3) f_{23}(x_2,x_3) 
	\le \norm{f_{12}}_\infty \norm{f_{13}}_\infty \norm{f_{23}}_\square.
	\]
\end{lemma}

\begin{proof}
	The above inequality is true if we fix any choice of $x_1$ on the LHS, due to the definition of $\norm{f_{23}}_\square$ in (\ref{eq:cut-fn}), and hence it remains true if we take the expectation over $x_1$.
\end{proof}

\begin{proof}[Proof of \cref{thm:sparse-c5-count}]
	For fixed $i,j \in \ZZ/5\ZZ$ with $i-j = \pm 1$ and $f_{ij} \colon V_i \times V_j \to [0,\infty)$, write
	\[
	f'_{ij}(x_i, x_j) =  f_{ij}(x_i, x_j)1_{S_{ij}}(x_i),
	\]
	where
	\[
	S_{ij} = \{x_i \in V_i : \EE_{x_j \in V_j} f_{ij}(x_i,x_j) \le \epsilon^{-2}\}.
	\]
Since $f_{ij}, g_{ij} \geq 0$, note that $\norm{f_{ij}}_1-\norm{g_{ij}}_1=\mathbb{E}_{x_i,x_j}(f_{ij}-g_{ij})$, which, by definition, is at most $\norm{f_{ij} - g_{ij}}_\square$. Hence, writing $\mu(\cdot)$ for the measure of a subset of $V_i$ (where $\mu(V_i) = 1$), we have 
	\[
	\mu(V_i \setminus S_{ij}) \epsilon^{-2} \le \norm{f_{ij}}_1 \le 
	\norm{g_{ij}}_1 + \norm{f_{ij} - g_{ij}}_\square
	\le 2.
	\]
	Therefore, $\mu(V_i \setminus  S_{ij}) \le 2\epsilon^2$ and 
	\begin{align*}
	\snorm{f'_{ij} - g_{ij}}_\square 
	&\le \snorm{(f_{ij} -g_{ij})1_{S_{ij} \x V_j}}_\square + \snorm{g_{ij} 1_{(V_i \setminus S_{ij}) \x V_j}}_\square 
	\\
	&\le \norm{f_{ij}-g_{ij}}_\square + \mu(V_i \setminus S_{ij})
	\\
	&\le 3\epsilon^2.
	\end{align*}
	
	Let $i,j,k$ be three consecutive elements of $\ZZ/5\ZZ$ (in ascending or descending order). 
	By \cref{lem:compose-cut}, 
	\[
	\norm{(f_{ij}-g_{ij}) \circ f'_{jk}}_\square \le \epsilon^{-2} \norm{f_{ij}-g_{ij}}_\square \le \epsilon^2
	\]
	and
	\[
	\norm{g_{ij} \circ (f'_{jk} - g_{jk})}_\square \le \norm{f'_{jk} - g_{jk}}_\square \le 3\epsilon^2.
	\]	
	Putting the above two inequalities together, we obtain
	\[
	\norm{f_{ij} \circ f'_{jk} - g_{ij} \circ g_{jk}}_\square
	\le 
	\norm{(f_{ij} - g_{ij}) \circ f'_{jk}}_\square  + 
	\norm{g_{ij} \circ (f'_{jk} - g_{jk})}_\square \le 4\epsilon^2.
	\]
	
	For every $A > 0$ and any function $h(z)$, define $h_{\le A}(z)=h(z)$ if $h(z) \leq A$ and $0$ otherwise. Then we have
	\begin{align*}
	\norm{(f_{ij} \circ f'_{jk})_{\le A} - g_{ij} \circ g_{jk}}_\square 
	&
	\le 
	\norm{f_{ij} \circ f'_{jk} - g_{ij} \circ g_{jk}}_\square + \norm{(f_{ij} \circ f_{jk})_{> A}}_1
	\\
	&\le 4\epsilon^2 + A^{-1} \norm{f_{ij} \circ f_{jk}}_2^2 
	\\&
	\le 4\epsilon^2 + C A^{-1}.
	\end{align*}
	In particular, setting $A = \epsilon^{-1}$, we obtain
	\begin{equation}\label{eq:f34-f'45}
		\norm{(f_{34} \circ f'_{45})_{\le \epsilon^{-1}} - g_{34} \circ g_{45}}_\square \le 5C \epsilon 
	\end{equation}
	and setting $A = \epsilon^{-2}$, we obtain
	\begin{equation} \label{eq:f32-f'21}
		\norm{(f_{32} \circ f'_{21})_{\le \epsilon^{-2}} - g_{32} \circ g_{21}}_\square \le 5C \epsilon^2.
	\end{equation}
	
	Repeatedly applying \cref{lem:dense-triangle} in the $3$rd, $4$th and $5$th inequalities, we have
	\begin{align*}
	\text{LHS of \cref{eq:sparse-c5-count}} 
	&\ge  \EE_{x_1,x_3,x_5} (f_{34} \circ f'_{45})(x_3,x_5) (f_{32} \circ f'_{21})(x_3,x_1) f_{51}(x_5,x_1) &&\text{[pointwise bound]}
	\\
	&\ge \EE_{x_1,x_3,x_5} (f_{34} \circ f'_{45})_{\le \epsilon^{-1}}(x_3,x_5) (f_{32} \circ f'_{21})_{\le \epsilon^{-2}}(x_3,x_1) f_{51}(x_5,x_1) &&\text{[pointwise bound]}
	\\
	&\ge \EE_{x_1,x_3,x_5} (f_{34} \circ f'_{45})_{\le \epsilon^{-1}}(x_3,x_5) (f_{32} \circ f'_{21})_{\le \epsilon^{-2}}(x_3,x_1) g_{51}(x_5,x_1) - \epsilon &&\text{[since $\norm{f_{51}-g_{51}}_\square \le \epsilon^4$]}
	\\
	& \ge \EE_{x_1,x_3,x_5} (f_{34} \circ f'_{45})_{\le \epsilon^{-1}}(x_3,x_5) (g_{32} \circ g_{21})(x_3,x_1) g_{51}(x_5,x_1) - 6C\epsilon && \text{[by \cref{eq:f32-f'21}]} 
	\\
	& \ge \EE_{x_1,x_3,x_5} (g_{34} \circ g_{45})(x_3,x_5) (g_{32} \circ g_{21})(x_3,x_1) g_{51}(x_5,x_1)- 11C \epsilon && \text{[by \cref{eq:f34-f'45}]}
	\\
	&= \text{RHS of \cref{eq:sparse-c5-count}}.
	&& \qedhere
	\end{align*}
\end{proof}

\section{Removal lemmas in sparse graphs}

Our proof of the sparse removal lemma uses, as a black box, the usual graph removal lemma (for dense graphs). We use the following formulation of the graph removal lemma allowing both vertex and edge weights. A standard sampling argument shows that this formulation (at least with a finite $V$) is equivalent to the more usual version without weights. Alternatively, the standard proof of the graph removal lemma using Szemer\'edi's regularity lemma can easily be amended to give the weighted version.

\begin{theorem}[Weighted graph removal lemma, dense setting] \label{thm:rem-dense}
	For every graph $F$ and $\epsilon >0$, there exists some $\delta > 0$ such that for a probability space $V$ and measurable symmetric $g \colon V \times V \to [0,1]$ with $t(F, g) \le \delta$, 
	there exists a measurable symmetric subset $A \subseteq V \times V$ such that 
	$t(F, g 1_A) = 0$ and 
	$\snorm{g - g 1_A}_1 \le \epsilon$.
\end{theorem}

Now we are ready to prove our main sparse removal lemma. We will first state and prove a version which highlights the hypotheses involved and then show that it implies both \cref{thm:345-removal,thm:5-cycle-rem}.

\begin{proposition}[Removal] \label{prop:dpc-removal}
	For every $\epsilon, K, C > 0$, there exist $n_0, M, \delta > 0$ such that if $n > n_0$, $p \in (C^{-1}n^{-1/2},1]$, and $G$ is an $n$-vertex graph satisfying
	\begin{enumerate}
		\item[(a)] (Dense pairs condition) for every partition of $V(G)$ into $M$ parts $V_1 \cup \cdots \cup V_M$, at most a total of $\epsilon p n^2/4$ edges of $G$ lie between pairs $(V_i, V_j)$ with $e(V_i, V_j) \ge K p \abs{V_i}\abs{V_j}$,
		\item[(b)] (Not too many $4$-cycles) $G$ has at most $(Cpn)^4$ copies of $C_4$, and
		\item[(c)] (Few 5-cycles) $G$ has at most $\delta p^5 n^5$ copies of $C_5$,
	\end{enumerate}
	then $G$ can be made $C_3$-free and $C_5$-free by removing at most $\epsilon p n^2$ edges.
\end{proposition}

\begin{proof}
The value of $\delta = \delta(\epsilon, K, C)$ will be given later in the proof using \cref{thm:rem-dense}. For now, we simply assume that it has a fixed value.

Write $V = V(G)$ and $f = p^{-1} 1_G \colon V \times V \to [0,p^{-1}]$ for the normalized edge-indicator function of $G$.  Hypothesis (a) implies that $G$ has at most $(K + \epsilon)pn^2/2$ edges. Hence, $\EE f \le K + \epsilon$.

Apply \cref{thm:reg-fn}, the sparse weak regularity lemma, to the function $f/K$ to obtain a partition $\cP$ of $V$ into at most $M = M(K, \epsilon, \delta)$ parts such that
\[
\norm{ (f - f_\cP)1_{f_\cP \le K}}_\square \le K \delta^4,
\]
which can be rewritten as
\[
\snorm{\wt f - \wt g}_\square \le K \delta^4,
\]
where
\[
\wt f = f 1_{f_\cP \le K}
\quad \text{and} \quad 
\wt g = \wt f_{\cP} = f_\cP 1_{f_\cP \le K}.
\]
Note that $\wt g$ can be obtained from $\wt f$ by averaging over pairs of parts of $\cP$.

By hypothesis (a), $G$ has at most $\epsilon pn^2/4$ edges in $\{f_\cP > K\} \subseteq V \times V$, since the latter is precisely the union of pairs $V_i \times V_j$ of $\cP$ with $e(V_i, V_j) > Kp\abs{V_i}\abs{V_j}$.
The graph obtained after removing these edges from $G$ is represented by $\wt f$.

Now we would like to use the 5-cycle counting lemma (\cref{thm:sparse-c5-count}) to deduce that $t(C_5, \wt g)$ must be small from the fact that $G$ has few 5-cycles. This is basically true, but one has to be a bit careful in the application of the $C_5$-counting lemma. The reason is that while we know that $G$ has few 5-cycles, we have not ruled out the possibility that 
the triangles of $G$ give rise to many homomorphic copies of $C_5$. 
We address this somewhat technical issue by splitting each part of $\cP$ arbitrarily into five nearly equal parts labeled by elements of $\ZZ/5\ZZ$ 
	and only considering 5-cycles where the $i$-th vertex is embedded into a part labeled by $i$, 
	so that all five vertices of the cycle are forced to be distinct. 
It is in this step that we need the $n > n_0$ hypothesis in the statement of the theorem, in order to guarantee that all five parts are nonempty.
Indeed, the theorem as stated is false without the $n > n_0$ hypothesis, 
	with $G = K_3$ being a counterexample. 
On the other hand, if we only wish to obtain a $C_5$-free graph, 
	then the $n > n_0$ hypothesis can be trivially removed by taking $\delta$ small enough that for $n \le  n_0$ the hypothesis (c) would already imply that $G$ is $C_5$-free.

For the details, we begin by removing some more edges. Indeed, if some part $V_i$ of the partition $\cP$ has at most $100$ vertices, then we remove all edges from $G$ with at least one vertex in $V_i$. We delete at most $100 M Kpn \le \epsilon pn^2/10$ edges this way, provided that $n > n_0(\epsilon, \delta, K)$ is sufficiently large. From now on, we may therefore assume that all parts of $\cP$ have more than 100 vertices.

Partition each $V_i \in \cP$ arbitrarily into five parts of nearly equal size (differing by at most 1), $V_i = V_{i}^{(1)} \cup \cdots \cup V_i^{(5)}$. 
Let $V^{(r)} = \bigcup_{i \in [M]} V_i^{(r)}$ for each $r \in [5]$. For each $r,s \in [5]$, write
\begin{equation} \label{eq:frs-grs}
f_{rs} = K^{-1} \wt f 1_{V^{(r)} \times V^{(s)}} \quad \text{and} \quad
g_{rs} = K^{-1} \wt g 1_{V^{(r)} \times V^{(s)}}
\end{equation}
(both functions $V \times V \to [0,\infty)$ with the latter taking values in $[0,1]$). Then
\[
\norm{ f_{rs} - g_{rs} }_\square 
= 
K^{-1} \snorm{ (\wt f - \wt g) 1_{V^{(r)} \times V^{(s)}} }_\square
\le K^{-1} \snorm{\wt f - \wt g}_\square \le \delta^4.
\]
By (b), $G$ has at most $(Cpn)^4$ copies of $C_4$.
We would actually like to upper bound the number of homomorphic copies of $C_4$ (i.e., closed walks of length $4$).
Let $r(x,y)$ denote the number of walks of length 2 from $x$ to $y$ in $G$. Then the number of homomorphic copies of $C_4$ in $G$ is at most (here we use the inequality $t^2 \le 4\binom{t}{2} + 1$ for all nonnegative integers $t$)
\[
n^2 + 
2 \sum_{\substack{x,y \in V(G) \\ x\ne y}} r(x,y)^2 \le 
n^2 + 2 \sum_{\substack{x,y \in V(G) \\ x\ne y}}\paren{4\binom{r(x,y)}{2} +1}
=3n^2 +  O\big((Cpn)^4\big) = O\big((Cpn)^4\big),
\]
where the final step uses the hypothesis $p \ge C^{-1}n^{-1/2}$.
After normalization, we obtain $t(C_4, f) = O (C^4)$. Hence,
\[
\norm{ f_{r-1,r} \circ f_{r,r+1}}_2^2 \le K^{-4} t(C_4, f) = O(C^4K^{-4}).
\]
Applying the 5-cycle counting lemma (\cref{thm:sparse-c5-count}) to the functions $f_{r,r+1}$ and $g_{r,r+1}$, with $r \in \ZZ/5\ZZ$, we obtain
\[
\EE_{x_1, \dots, x_5 \in V} \prod_{r=1}^5 f_{r,r+1}(x_r, x_{r+1})
\ge \EE_{x_1, \dots, x_5 \in V} \prod_{r=1}^5 g_{r,r+1}(x_r, x_{r+1}) - O( C^4K^{-4}\delta),
\]
which, by \cref{eq:frs-grs}, can be rewritten as
\begin{equation}
	\label{eq:c5-disjoint}
\EE_{x_1, \dots, x_5 \in V} \prod_{r=1}^5 \wt f(x_r, x_{r+1}) 1_{V^{(r)}}(x_r)
\ge \EE_{x_1, \dots, x_5 \in V} \prod_{r=1}^5 \wt g(x_r, x_{r+1}) 1_{V^{(r)}}(x_r) - O( C^4 K \delta).
\end{equation}
Each $V_i$ has size at least 100, so each of its 5 parts $V_i^{(r)}$, $r \in [5]$, has at least a 1/6-fraction of the vertices. 
Thus, $\sabs{V^{(r)}} \ge \abs{V}/6$ for each $r \in[5]$. 
Note that $\wt g$ is constant on each $V_i \times V_j$. 
It follows that the RHS of \cref{eq:c5-disjoint} is at least $6^{-5} t(C_5, \wt g) - O(C^4K\delta)$. 
On the other hand, the LHS of \cref{eq:c5-disjoint} is $p^{-5}n^{-5}$ times the number of 5-cycles in $G$ with the $r$-th vertex in $V^{(r)}$ for each $r \in [5]$
 and so the LHS of \cref{eq:c5-disjoint} is at most $\delta$ by hypothesis (c) that $G$ has at most $\delta n^5p^5$ copies of $C_5$.
Putting these two bounds together, we obtain
\[
\delta \ge 6^{-5} t(C_5, \wt g) - O( C^4K\delta).
\]
Thus, $t(C_5, \wt g) \lesssim (1 + C^4K)\delta$.
Now apply \cref{thm:rem-dense}, the graph removal lemma, for $C_5$ and choose $\delta = \delta(\epsilon, K, C)$ small enough so that the above upper bound on $t(C_5, \wt g)$ guarantees that there exists some symmetric $A \subseteq V \times V$, which is a union of pairs of parts of the partition $\cP$, such that $\norm{\wt g - \wt g 1_A}_1 \le \epsilon/3$ and $t(C_5, \wt g 1_A) = 0$. 
Note that we need to apply the dense removal lemma to the weighted graph with vertices being the parts of $\cP$ and vertex weights proportional to the sizes of the parts so that the dense removal lemma outputs an $A$ of the desired form.

So $t(C_5, \wt g 1_A) = 0$ and, consequently, $t(C_3, \wt g 1_A) = 0$ as well. Thus, $t(C_5, \wt f 1_A) = 0$ and $t(C_3, \wt f 1_A) = 0$, since if $f$ takes some positive values on $V_i \times V_j$, then so must $g$. Since $\wt g$ is obtained from $\wt f$ by averaging over each pair of parts of $\cP$,
\[
\snorm{\wt f - \wt f 1_A}_1 = \snorm{\wt g - \wt g 1_A}_1 \le \epsilon/3.
\]
To conclude, the graph $G$ can be made $C_3$-free and $C_5$-free by removing at most $\epsilon pn^2$ edges: at most $\epsilon pn^2/4$ edges between pairs of parts $(V_i, V_j)$ with $e(V_i, V_j) \ge K p\abs{V_i}\abs{V_j}$, at most $\epsilon p n^2/10$ edges with at least one endpoint in some part with at most 100 vertices, and at most $\epsilon p n^2/3$ edges outside the set $A$ in the last step above.
\end{proof}

Recall the following equivalent way of stating \cref{thm:345-removal}:
\begin{quote}
	For every $\epsilon > 0$, there exist $n_0$ and $\delta> 0$ such that, 
	for every $n \ge n_0$, every $n$-vertex graph with 
	at most $\delta n^{5/2}$ copies of $C_5$ 
	and 
	at most $\delta n^2$ copies of $C_4$ 
	can be made $C_3$-free and $C_5$-free 
	by deleting at most $\epsilon n^{3/2}$ edges.
\end{quote}

\begin{proof}[Proof of \cref{thm:345-removal}]
	Let $G$ be an $n$-vertex graph with
	at most $\delta n^{5/2}$ copies of $C_5$ 
	and 
	at most $\delta n^2$ copies of $C_4$.
	Applying \cref{lem:c4-dense-pairs},
	we see that for every partition of $V(G)$ into $M$ parts $V_1 \cup \cdots \cup V_M$, 
	the number of edges of $G$ that lie  between ``dense pairs'' $(V_i, V_j)$ with $e(V_i, V_j) \ge K n^{-1/2} \abs{V_i}\abs{V_j}$ is
		$O(n^{3/2}/K + M^2 n + M^{1/4} \delta^{1/4} n^{3/2})$.
	For $p = n^{-1/2}$, this is at most $\epsilon p n^2/4$ 
		provided that 
			$K$ is a sufficiently large constant times $1/\epsilon$,
			$\delta$ is a sufficiently small constant times $\epsilon^4/M$,
			and $n$ is sufficiently large (depending on $\epsilon$ and $M$).
	Thus, condition (a) of \cref{prop:dpc-removal} is satisfied.
	
	Condition (b) of \cref{prop:dpc-removal} is also automatically satisfied with $C = 1$ provided that $\delta < 1$, as is Condition (c) for $\delta$ sufficiently small. It thus follows from \cref{prop:dpc-removal} that one can remove at most $\epsilon n^{3/2}$ edges to make $G$ both $C_3$-free and $C_5$-free.
\end{proof}

\cref{thm:5-cycle-rem} is the $5$-partite version of the sparse $C_5$-removal lemma, 
	where we instead assume that there are at most $\delta n^2$ copies of $C_4$ between any two consecutive vertex sets.

\begin{proof}[Proof of \cref{thm:5-cycle-rem}]
	The proof is nearly identical to the proof of \cref{thm:345-removal}.
	Set $p = n^{-1/2}$ as earlier.
	As before, we apply \cref{lem:c4-dense-pairs} to each bipartite graph $V_i \times V_{i+1}$ to yield that at most $\epsilon pn^2/4$ edges lie between dense pairs for any partition into at most $M$ parts. This verifies condition (a) of \cref{prop:dpc-removal}.
	
	To verify condition (b) of \cref{prop:dpc-removal}, note that the number of 4-cycles spanning three parts $V_{i-1}, V_i, V_{i+1}$ is given by (the sums are taken over all unordered pairs of distinct vertices $x,y \in V_i$ and $\codeg_j(x,y)$ is the number of common neighbors of $x$ and $y$ in $V_j$)
	\[
	\sum_{x \ne y \in V_i} \codeg_{i-1}(x,y) \codeg_{i+1}(x,y)
	\le \left(\sum_{x \ne y \in V_i} \codeg_{i-1}(x,y)^2\right)^{1/2} \left(\sum_{x\ne y \in V_i} \codeg_{i+1}(x,y)\right)^{1/2},
	\]
	where the inequality follows from Cauchy--Schwarz. Using that $t^2 \le 1 + 4\binom{t}{2}$ for all nonnegative integers $t$ and provided that $\delta < 1/8$, we have 
	\[
	\sum_{x \ne y \in V_i} \codeg_{i-1}(x,y)^2
	\le \binom{n}{2} + 4 \sum_{x\ne y \in V_i} \binom{\codeg_{i-1}(x,y)}{2} \le \binom{n}{2} + 4\delta n^2 \le n^2,
	\]
	where we apply hypothesis (a) of \cref{thm:5-cycle-rem} that there are at most $\delta n^2$ $C_4$'s between $V_{i-1}$ and $V_i$. Likewise,
	\[
	\sum_{x \ne y \in V_i} \codeg_{i+1}(x,y)^2 \le n^2.
	\]
	Hence, the number of 4-cycles spanning the vertex sets $V_{i-1}, V_i, V_{i+1}$ is at most $n^2$ for each $i$. Thus, condition (b) of \cref{prop:dpc-removal} is satisfied with $C = 1$.
	
	Finally, condition (c) of \cref{prop:dpc-removal} is satisfied due to hypothesis (b) of \cref{thm:5-cycle-rem}.
	
	The $C_5$-removal claim thus follows. 
	Note that the ``sufficiently large $n$'' condition is  superfluous here, since we can always make $\delta$ smaller to take care of the finite number of potentially exceptional values of $n$.
\end{proof}

\section{Removal lemma corollaries} \label{sec:pf-rem-cor}

Here we prove \cref{cor:o-C5} of the sparse removal lemma~\cref{thm:345-removal}, saying that an $n$-vertex graph with $o(n^2)$ $C_5$'s can be made triangle-free by removing $o(n^{3/2})$ edges.

In each of the following proofs, we let $G$ be the graph in the statement and $G'$ be a subgraph of $G$ whose edge set is the union of a maximal collection of edge-disjoint triangles in $G$. In order to show that $G$ can be made triangle-free by deleting $o(n^{3/2})$ edges, it is sufficient to show that $G'$ has $o(n^{3/2})$ edge-disjoint triangles, which is in turn equivalent to showing that $G'$ can be made triangle-free by deleting $o(n^{3/2})$ edges. This last statement is what we will show.

Let us first prove a corollary of \cref{thm:345-removal} that may be of independent interest.
The \emph{house} graph is depicted below.
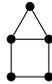
\[
\begin{tikzpicture}
		[
			every node/.style={draw, circle, black, fill, inner sep = 0pt, minimum width = 3pt},
			scale=.5,
		]
    \node (00) at (0,0) {};
    \node (01) at (0,1) {};
    \node (10) at (1,0) {};
    \node (11) at (1,1) {};
    \node (t)  at (0.5,1.87) {};
    \draw (01)--(00)--(10)--(11)--(01)--(t)--(11);
\end{tikzpicture}
\]

\begin{corollary} \label{cor:house}
	An $n$-vertex graph with 
	$o(n^2)$ $C_4$'s that extend to houses 
	and
	$o(n^{5/2})$ $C_5$'s 
	can be made triangle-free by deleting $o(n^{3/2})$ edges.
\end{corollary}

\begin{proof}
	It is easy to check that
	in an edge-disjoint union of triangles, every $C_4$ extends to a house. 
	Following the notation above, 
		since each $C_4$ in $G'$ extends to a house in $G'$, 
		$G'$ has $o(n^2)$ $C_4$'s. 
	Moreover, $G'$ has $o(n^{5/2})$ $C_5$'s, since the same is true in $G$.
	Applying \cref{thm:345-removal} then yields that $G'$ can be made triangle-free by removing $o(n^{3/2})$ edges. 
\end{proof}

\begin{proof}[Proof of \cref{cor:o-C5}]
	Since $G'$ is an edge-disjoint union of triangles, every $C_4$ in $G'$ extends to a house, which contains a $C_5$.
	Moreover, each $C_5$ in $G'$ can arise from at most five different $C_4$'s in this way. 
	Since $G'$ has $o(n^2)$ $C_5$'s, we see that $G'$ has $o(n^2)$ $C_4$'s.
	Thus, \cref{thm:345-removal} implies that $G'$ can be made triangle-free by removing $o(n^{3/2})$ edges.
\end{proof}

\section{Number-theoretic applications} \label{sec:pf-nt}

Suppose that $a_1, \dots, a_5$ are fixed nonzero integers summing to zero and $X_1, \dots, X_5$ are subsets of $[n]$ such that each $X_i$ has $o(n)$ nontrivial solutions to $x_1 + x_2 = x_3 + x_4$ and there are $o(n^{3/2})$ solutions to $a_1 x_1 + \cdots + a_5 x_5 = 0$ with $x_1 \in X_1, \dots, x_5 \in X_5$. Then \cref{thm:transl-invar}, which we now prove, says that we can remove $o(\sqrt{n})$ elements from each $X_i$ to remove all solutions to $a_1 x_1 + \cdots + a_5 x_5 = 0$ with $x_1 \in X_1, \dots, x_5 \in X_5$.

\begin{proof}[Proof of \cref{thm:transl-invar}]
	Embed $X_1, \dots, X_5$ into $\ZZ/N\ZZ$, 
	where $N$ is the smallest integer greater than $(|a_1| + \cdots + |a_5|)n$ (to avoid wraparound issues)
	which is coprime to each of $a_1, \dots, a_5$.
	We consider the 5-partite graph $G$ with vertex sets $V_1, \dots, V_5$, each with elements indexed by $\ZZ/N\ZZ$, and
	edges $(s, s + a_i x_i) \in V_i \times V_{i+1}$ for  all $i$ (mod 5), $s \in \ZZ/N\ZZ$, and $x_i \in X_i$.
	
	Now we verify that $G$ satisfies the hypotheses of the sparse 5-cycle removal lemma, \cref{thm:5-cycle-rem}:
	
	(a) All $4$-cycles lying between $V_i$ and $V_{i+1}$ are of the form 
	\[
	s, \quad  s + a_i x_1, \quad s + a_i(x_1 - x_2), \quad s+a_i(x_1 - x_2 + x_3)
	\]
	and we must have $s = s + a_i(x_1 - x_2 + x_3 - x_4)$ for some $x_4$ to close off the cycle. In order for the $4$-cycle to have distinct vertices, $(x_1, x_2, x_3, x_4) \in X_i^4$ must be a nontrivial solution to $x_1 + x_3 = x_2 + x_4$. But there are $o(N)$ such 4-tuples, while the choice of $s \in \ZZ/N\ZZ$ is arbitrary, so there are $o(N^2)$ $4$-cycles between each pair $(V_i, V_{i+1})$.
	
	(b) The number of 5-cycles in $G$ equals $N$ times the number of solutions to $a_1 x_1 + \cdots + a_5 x_5 = 0$ with $x_1 \in X_1$, \dots, $x_5 \in X_5$, so there are $o(N^{5/2})$ $C_5$'s.
	
	Thus, by \cref{thm:5-cycle-rem},
	$G$ can be made $C_5$-free by removing $o(N^{3/2})$ edges. 
	In each $X_i$, we now remove the element $x_i$ if at least $N/5$ edges of the form $(s, s + a_i x_i)$ have been removed. Since we removed $o(N^{3/2})$ edges from $G$, we remove  $o(\sqrt{N}) = o(\sqrt{n})$ elements from each $X_i$. Let $X'_i$ denote the remaining elements of $X_i$. For any solution to $a_1 x_1 + \dots + a_5x_5 = 0$ with $x_1 \in X_1, \dots, x_5 \in X_5$, consider the $N$ edge-disjoint $5$-cycles in the graph $G$ that arise from this solution. We must have removed at least one edge from each of these $5$-cycles and so we must have removed at least $N/5$ edges of the form $(s, s + a_i x_i)$ for some $i$, which implies that $x_i \notin X'_i$. There must therefore be no solution to $a_1 x_1 + \dots + a_5x_5 = 0$ with $x_1 \in X'_1, \dots, x_5 \in X'_5$, as required.
	\end{proof}

Recall that the following lemma, saying that a Sidon set has $O(n)$ solutions to $a_1 x_1 + a_2 x_2 + a_3 x_3 + a_4 x_4 = 0$ for any nonzero integers $a_1, \dots, a_4$, was needed to derive \cref{thm:avoid-two-eqn2} from \cref{thm:transl-invar}.

\begin{proof}[Proof of \cref{lem:non-distinct}]
Writing
\[
\wh 1_X(t) = \sum_{x \in X} e^{2\pi i xt},
\]
we have, by a standard Fourier identity (easy to see by expansion), that
\begin{align*}
&\hspace{-2em} \abs{\set{(x_1, x_2,x_3,x_4)\in X^4 : a_1 x_1 + a_2 x_2 + a_3 x_3 + a_4x_4 = 0}}
\\
&= \int_0^1 \wh 1_X(a_1t) \wh 1_X(a_2t) \wh 1_X(a_3t) \wh 1_X(a_4t) \, dt
\\
&\le \prod_{j=1}^4 \paren{\int_0^1 |\wh 1_X(a_jt)|^4 \, dt}^{1/4}.
\end{align*}
Moreover,
\[
\int_0^1 |\wh 1_X(a_jt)|^4 \, dt 
= \abs{\set{(x_1, x_2,x_3,x_4)\in X^4 : x_1 + x_2 = x_3 + x_4}}
= 2\abs{X}^2 - \abs{X} = O(n),
\]
since $X \subset [n]$ is a Sidon set.
\end{proof}

\section{Some constructions} \label{sec:constructions}

In the previous section, we deduced our number-theoretic results by starting with a set of integers avoiding solutions to certain equations and building an associated graph to which we could apply our removal lemma. We now use this same idea to prove \cref{prop:C5-construction}, which asserts the existence of an $n$-vertex $C_4$-free graph with $n^{3/2 - o(1)}$ edges where every edge lies in exactly one $5$-cycle. 

\begin{proof}[Proof of \cref{prop:C5-construction} ]
First, we note that there exists a set $X \subseteq [n]$ with $|X| \ge e^{- O(\sqrt{\log n})} n^{1/2}$ with 
\begin{enumerate}
	\item no nontrivial solutions to the equation $a(x_1 - x_2) = b(x_3 - x_4)$ for all $a,b \in \{1,2,3,4,10\}$ \\
	(here a solution is called trivial if $x_1 = x_2$ and $x_3 = x_4$ or $a=b$, $x_1 = x_3$ and $x_2 = x_4$) and
	\item no nontrivial solutions to the equation $x_1 + 2x_2 + 3x_3 + 4x_4 = 10x_5$
	\\
	(here a solution is called trivial if $x_1 = \cdots  = x_5$).
\end{enumerate}
The existence of such a set follows from two ingredients, both essentially noted by Ruzsa~\cite{Ruz93}.
Indeed, sets of size $e^{-O( \sqrt{\log n})} n^{1/2}$ satisfying the first property can be constructed through a minor modification of \cite[Theorem 7.3]{Ruz93}, as noted in \cite{CT14}.
Moreover, a set of size $e^{-O(\sqrt{\log n})}n$ satisfying the second property exists by a standard adaptation~\cite[Theorem 2.3]{Ruz93} of Behrend's construction~\cite{Beh46}. 
Since the second property is translation invariant,
taking the intersection of a random translation of the second set with the first set gives a set with the claimed size satisfying both properties.	

Let $N = 60n + 1$. Let $V_1, \dots, V_5$ be vertex sets each with vertices indexed by $\ZZ/N\ZZ$. Add a $5$-cycle $(s, s+ x, s+3x, s+ 6x, s+10x) \in V_1 \times \cdots \times V_5$ to the graph for each $s \in \ZZ/N\ZZ$ and $x \in X$. 
This construction is similar to the construction used in the proof of \cref{thm:transl-invar}.

We now show that this graph has all of the required properties. First note that the cycles are edge-disjoint, since if, for instance, $s + 3x  = s' + 3 x'$ and $s + 6x = s' + 6 x'$, then $s = s'$ and $x = x'$, where we used that $N$ is coprime to $\{1, 2, 3, 4, 10\}$. Moreover, there are no further $C_5$'s in this graph, since any other such 5-cycle would give a nontrivial solution to the equation $x_1 + 2 x_2 + 3 x_3 + 4 x_4 = 10 x_5$ in $X$.

To see that there are no $C_4$'s, note that there are no $C_4$'s between any pair of vertex sets since $X$ is Sidon. Moreover, there are no $C_4$'s across three vertex sets, since, for instance, any $4$-cycle spanning $V_2, V_3, V_4$ would induce a nontrivial solution to the equation $2(x_1 - x_2) = 3(x_3 - x_4)$.
\end{proof}

\begin{remark}
	We do not know how to show that the exponent in \cref{cor:C5-free} is best possible. Recall the statement, that every $n$-vertex $C_5$-free graph can be made triangle-free by deleting $o(n^{3/2})$ edges. To match this with a lower bound,
	we would like to construct a tripartite graph between sets of order $n$ which is the union of $n^{3/2 - o(1)}$ edge-disjoint triangles but containing no $C_5$. 
	Following the strategy above, this would require a set $X \subseteq [n]$ of size $n^{1/2 - o(1)}$ which satisfies properties similar to (1) and~(2), one particular case being that there should be no nontrivial solutions to the equation $x + 2y + 2z = 2u + 3w$. 
	At present, we do not know how to construct such a set satisfying even this latter property on its own. In fact, it is entirely plausible that no such set exists.
\end{remark}

Finally, we prove \cref{prop:construction-few-C5},
    which says that there exist $n$-vertex graphs with $o(n^{2.442})$ $C_5$'s that cannot be made triangle-free by deleting $o(n^{3/2})$ edges.
    
\begin{proof}[Proof of \cref{prop:construction-few-C5}]
    Let $G$ be the $m$-th tensor power of a triangle.
    In other words, its vertices can be labeled by $\FF_3^m$ and two vertices are adjacent iff they differ in every coordinate. 
    Note that every edge of $G$ lies in a unique triangle and this graph has $6^{m-1}$ such triangles.
    
    Write $\hom(H, G)$ for the number of homomorphisms from $H$ to $G$. 
    The number of closed walks of length $k$ in $C_3$ is equal to the $k$-th moment of the eigenvalues of the adjacency matrix of $C_3$ and so  $\hom(C_k, C_3) = 2^k + 2(-1)^k$. 
    Thus,
    \[
    \hom(C_3, G) = \hom(C_3, C_3)^m = 6^m,
    \]
    \[
    \hom(C_4, G) = \hom(C_4, C_3)^m = 18^m,
    \]
    and
    \[
    \hom(C_5, G) = \hom(C_5, C_3)^m = 30^m.
    \]
    
    Let $G'$ be the graph obtained from $G$ by keeping every triangle independently with probability $p  = (\sqrt{3}/2)^m$, so that in expectation $p 6^{m-1} = \Theta(3^{3m/2})$ triangles remain in $G'$.
    
    To estimate the expected number of $C_5$ in $G'$, note that
    the number of $C_5$ that intersect exactly four triangles of $G$ is $O(18^m)$, since every such $C_5$ extends to a house, which contains a $C_4$.
    Furthermore, it is impossible for a $C_5$ in $G$ to intersect at most three triangles, as every edge of $G$ is contained in exactly one triangle. Thus, the expected number of $C_5$ in $G'$ is $O(p^5 30^m + p^4 18^m) = O((3^{7/2}\cdot5/16)^m) = o(3^{2.442m})$.
    
    Therefore, by Markov's inequality, we see that with positive probability $G'$ is a graph on $n=3^m$ vertices with $o(n^{2.442})$ $C_5$'s which is an edge-disjoint union of $\Theta(n^{3/2})$ triangles.
\end{proof}

\section*{Acknowledgments}
We would like to thank Jacques Verstra\"ete and J\'ozsef Solymosi for helpful comments.

\appendix

\section{Proof of the sparse weak regularity lemma} \label{sec:pf-reg}

Here we prove \cref{thm:reg-fn}, our sparse weak regularity lemma.
As with nearly all proofs of regularity lemmas, 
	we keep track of some ``energy'' function and show, 
	by using a certain defect inequality, 
	that the energy must increase significantly at every iteration of a partition refinement process. 
	As in Scott~\cite{Sco11}, the energy of a partition will be defined using the convex function
\[
\Phi(x) := \begin{cases}
	x^2 & \text{if } 0 \le x \le 2 \\
	4x - 4 & \text{if } x > 2.
\end{cases}
\]

The defect inequality is now captured by the following lemma.

\begin{lemma} \label{lem:defect}
	Every real-valued random variable $X \ge 0$ with expectation $\mu \le 1$ satisfies
	\[
		\frac14 (\EE|X - \mu|)^2 \le \EE \Phi(X) - \Phi(\EE X).
	\]
\end{lemma}

\begin{proof}
	Let $p_1 = \PP(X < \mu)$ and $p_2 = \PP(X \ge \mu)$. Let $\mu_1 =  \EE[X | X < \mu]$  and $\mu_2 = \EE[X | X \ge \mu]$ (setting $\mu_1 = \mu$ if $p_1 = 0$). Note that $\mu=p_1 \mu_1+p_2 \mu_2$. We consider two cases.
	
	\vspace{1mm}
	\noindent
	\textbf{Case 1:} $\mu_2 \le 2$. By the convexity of $\Phi$,
	\begin{align*}		
	\EE \Phi(X) - \Phi(\EE X) &\ge p_1 \Phi(\mu_1) + p_2 \Phi(\mu_2) - \Phi(\mu)  
	= p_1 \mu_1^2 + p_2 \mu_2^2 - \mu^2 
	\\
	&= p_1(\mu - \mu_1)^2 + p_2 (\mu_2 - \mu)^2 
	\ge  (p_1 (\mu - \mu_1) + p_2(\mu_2 - \mu))^2 
	=	(\EE|X - \mu|)^2.
	\end{align*}
	
	\noindent
	\textbf{Case 2:} $\mu_2 \ge 2$. Let $\Phi_\mu (x) := \Phi(x) - 2\mu x + \mu^2$, which is convex. 
	Indeed, $\Phi_\mu$ is a quadratic function on $[0,2]$ with minimum at $\mu$ and is linear on $[2,\infty)$.
	Since $0 \le \mu_1 \le \mu \le 1$ and $\mu_2 \ge 2$, we have 
	$\Phi_\mu(\mu_1) \le \Phi_\mu(0)  = \mu^2 \le 1 \le (2-\mu)^2 = \Phi_\mu(2) \le  \Phi_\mu(\mu_2)$. 
	Thus,
	\begin{align*}
	\EE \Phi(X) - \Phi(\EE X) &= \EE\Phi_\mu(X) \ge p_1 \Phi_\mu (\mu_1) + p_2 \Phi_\mu(\mu_2) \ge \Phi_\mu(\mu_1) 
	= (\mu_1 - \mu)^2 
	\\
	&\ge \frac14 ( 2p_1 (\mu - \mu_1))^2 
	=\frac14 ( p_1 (\mu - \mu_1) + p_2 (\mu_2 - \mu))^2 
	=  \frac{1}{4} (\EE |X - \mu|)^2. \qedhere 
	\end{align*}
\end{proof}

\begin{lemma} \label{lem:defect-global}
	Let $V$ be a probability space and $f \colon V \times V \to [0, \infty)$ be a measurable symmetric function. Let $\cP$ and $\cQ$ be measurable partitions of $V$ such that $\cQ$ refines $\cP$. Then
	\[
		\EE [\Phi \circ f_\cQ] - \EE [\Phi \circ f_\cP]
		\ge \frac14 (\EE [\abs{f_\cQ - f_\cP} 1_{f_\cP \le 1}])^2.
	\]
\end{lemma}

\begin{proof}
	For every pair $(U,W)$ of parts of $\cP$, 	the function $f_{\cP}$ is constant on $U \x W$ with value the average of $f$ on $U \x W$, which is the same as the average of $f_\cQ$ on $U \x W$ since $\cQ$ is a refinement of $\cP$.
	So, by the convexity of $\Phi$, we always have $\EE_{U \times W} [\Phi \circ f_\cQ] \ge  \EE_{U \times W} [\Phi \circ f_\cP]$. 
	Furthermore, by applying \cref{lem:defect} on those $U \times W$ where $f_{\cP} \le 1$, we deduce that
	\[
	\EE_{U \times W} [\Phi \circ f_\cQ] -  \EE_{U \times W} [\Phi \circ f_\cP] \ge \frac14 (\EE_{U\times W} [\abs{f_\cQ - f_\cP} 1_{f_\cP \le 1}])^2
	\] 
	for all $(U, W) \in \cP \times \cP$. Finally, summing the above inequality over all pairs $(U,W)$ and applying the Cauchy--Schwarz inequality, we have
	\begin{align*}
	\EE [\Phi \circ f_\cQ] - \EE [\Phi \circ f_\cP]
	&= \sum_{U,W \in \cP} \mu(U) \mu(W) \paren{ \EE_{U \times W} [\Phi \circ f_\cQ] -  \EE_{U \times W} [\Phi \circ f_\cP]}
	\\
	&\ge \frac14 \sum_{U,W \in \cP} \mu(U) \mu(W) (\EE_{U\times W} [\abs{f_\cQ - f_\cP} 1_{f_\cP \le 1}])^2
	\\
	&\ge \frac14 \paren{\sum_{U,W \in \cP} \mu(U) \mu(W)  \EE_{U\times W} [\abs{f_\cQ - f_\cP} 1_{f_\cP \le 1}]}^2
	\\
	&= \frac14 (\EE [\abs{f_\cQ - f_\cP} 1_{f_\cP \le 1}])^2. \qedhere 
	\end{align*}
\end{proof}

Now we prove the sparse weak regularity lemma. Recall the statement, that, given $\epsilon > 0$ and $f \colon V \times V \to [0,\infty)$, there exists a partition $\cP$ of $V$ into at most $2^{32 \EE f/\epsilon^2}$ parts such that 
\begin{equation} \label{eq:reg-f}
\norm{ (f - f_\cP) 1_{f_\cP \le 1}}_\square \le \epsilon.
\end{equation}

\begin{proof}[Proof of \cref{thm:reg-fn}]
	Starting with the trivial partition $\cP_0$ of $V$ (i.e., consisting of a single part), consider the following iterative process: for each $i \ge 0$, if \cref{eq:reg-f} is satisfied for $\cP = \cP_i$, then we stop the iteration; otherwise, by the definition of the cut norm, there exist measurable subsets $A,B \subset V$ such that $\abs{\EE[(f - f_\cP) 1_{f_\cP \le 1} 1_{A\times B}]} > \epsilon$  and we set $\cP_{i+1}$ to be the common refinement of the partition $\cP_i$ and the four-part partition determined by $A$ and $B$. Write 
	\[
	\Phi(\cP) := \EE[\Phi \circ f_\cP] = \EE_{x,y \in V}[\Phi(f_{\cP})]
	\]
	for the ``energy'' of the partition $\cP$.
	
	With $\cP = \cP_i$, $\cQ = \cP_{i+1}$, and $A,B \subset V$ as above, we have, by \cref{lem:defect-global}, that
	\[
		\Phi(\cQ) - \Phi(\cP)
		\ge \tfrac14 (\EE [\abs{f_\cQ - f_\cP} 1_{f_\cP \le 1}])^2.
	\]
		Since $A$ and $B$ are unions of parts of $\cQ$, we have
	\[
	\EE [\abs{f_\cQ - f_\cP} 1_{f_\cP \le 1}]
	\ge 
	\abs{ \EE [(f_\cQ - f_\cP) 1_{f_\cP \le 1} 1_{A \times B}] }
	= 
	\abs{ \EE [(f - f_\cP) 1_{f_\cP \le 1} 1_{A \times B}] } 
	> \epsilon. 
	\]
	Thus, for every $i \ge 0$, one has
	\[
	 \Phi(\cP_{i+1}) - \Phi(\cP_i) \ge \tfrac14 \epsilon^2.
	\]
	On the other hand, since $0 \le \Phi(x) \le 4x$ for all $x \ge 0$, for every partition $\cP$ of $V$, we have
	\[
	 \Phi(\cP) \le 4 \EE f_{\cP} = 4 \EE f.
	\]
	Thus, the iteration must terminate after at most $16 \EE f / \epsilon^2$ steps, at which point \cref{eq:reg-f} is satisfied. The number of parts increases by a factor of at most 4 in each iteration, so the final number of parts is at most $2^{32 \EE f/\epsilon^2}$.
\end{proof}

\section{Connecting Brown--Erd\H{o}s--S\'os to the extremal problem for girth} \label{sec:pf-girth}

Here we prove \cref{prop:girth}, which says that $f_r(n,(r-1)e,e) = h_r(n,e)$ for sufficiently large $n \ge n_0(r,e)$. Recall that $f_r(n, v, e)$ is the maximum number of edges in an $n$-vertex $r$-graph
without a $(v,e)$-configuration, 
a subgraph with $e$ edges and at most $v$ vertices.\ Moreover, $h_r(n,g)$ is the maximum number of edges in an $n$-vertex $r$-graph of girth greater than $g$.

\begin{lemma}\label{lem:girth-cycle}
An $r$-graph $H$ on $n$ vertices without an $((r-1)e,e)$-configuration and with a cycle of length at most $e$ has at most $\max\{ e-1, n/(r-1) \}$ edges.  
\end{lemma}

\begin{proof}
By assumption, $H$ contains a cycle $C$ of length $\ell \leq e$. 
The vertices of this cycle span at most $(r-1)\ell$ vertices.
Growing the connected component containing $C$ one edge at a time, 
	we see that each additional edge adds at most $r-1$ new vertices. Let $P_0$ be the connected component containing $C$. 
In particular, if $P_0$ has at least $e$ edges, then it would contain $e$ edges spanning at most $(r-1)e$ vertices, 
	contradicting that $H$ has no $((r-1)e,e)$-configuration. Hence, $P_0$ has fewer than $e$ edges.
Let $P_1,\dots,P_j$ denote the remaining connected components of $H$. For $0 \leq i \leq j$, let $n_i$ and $e_i$ denote the number of vertices and edges, respectively, of $P_i$. 
Let $p_i=(r-1)e_i-n_i$ and assume, by reordering if necessary, that $p_1 \geq \cdots \geq p_j$. 
By construction, $p_0 \geq 0$ and $e_0 < e$. 

We may assume that $H$ has at least $e$ edges, as otherwise we are done.
Let $k$ be the largest index such that $e_0+\cdots+e_k<e$ and so $e_0+\cdots+e_k+e_{k+1} \geq e$, as $H$ has at least $e$ edges.
Let $e'=e-(e_0+\cdots+e_k)$. Let $Q$ be a connected subset of $P_{k+1}$ with $e'$ edges formed by starting with a single edge in $P_{k+1}$ and adding edges one at a time, keeping the resulting subset connected, until we get exactly $e'$ edges. By construction, $Q$ has at most $(r-1)e'+1$ vertices. 
Let $U_0$ consist of the union of the connected components $P_0,\ldots,P_k$ together with $Q$, 
	so that $U_0$ has $e$ edges and at most $\left((r-1)e_0-p_0\right)+\cdots +\left((r-1)e_k-p_k\right)+(r-1)(e-(e_0+\cdots+e_k))+1=(r-1)e+1-p_0-\cdots-p_k$ vertices.
As $H$ has no $((r-1)e,e)$-configuration, we must have $p_0+\cdots+p_k \leq 0$. 
As $p_1 \geq \cdots \geq p_j$, it follows that $p_0+p_1+\cdots+p_j \leq 0$. Equivalently, the number of edges in $H$ is at most $n/(r-1)$. 
\end{proof}

\begin{lemma} \label{lem:theta}
	For $r \geq 2$ and $g \geq 2$, there exists $n_0(r, g)$ such that $h_r(n,g) > n/(r-1)$ for all $n \geq n_0(r, g)$.
\end{lemma}

\begin{proof}
	Construct an $r$-graph by fixing two vertices $u$ and $v$ and adding edge-disjoint ``paths'' between $u$ and $v$, where each ``path'' consists of $\lfloor g/2 \rfloor + 1$ edges, the first containing $u$, the last containing $v$, and where consecutive edges share exactly one vertex. Add as many paths as one can without exceeding $n$ total vertices. Each new path uses $(\lfloor g/2 \rfloor + 1)(r-1) - 1$ new vertices (not counting $u$ and $v$). The resulting $r$-graph has girth greater than $g$ and 
	\[
	h_r(n,g) \ge \left\lfloor \frac{n-2}{(\lfloor g/2 \rfloor + 1)(r-1) - 1} \right\rfloor (\lfloor g/2 \rfloor + 1) 
	\]
	edges, which exceeds $n/(r-1)$ for sufficiently large $n$.
\end{proof}

\begin{proof}[Proof of \cref{prop:girth}]
If an $r$-graph has girth greater than $e$,
	then every subset of $e$ edges contains no cycle and hence spans more than $(r-1)e$ vertices.
	Thus, the $r$-graph has no $((r-1)e,e)$-configuration and $h_r(n,e) \leq f_r(n,(r-1)e,e)$ for all $n$.

Conversely, by the previous paragraph and Lemma \ref{lem:theta}, we have that for sufficiently large $n$ the largest $n$-vertex $r$-graph with no $((r-1)e,e)$-configuration 
has more than $n/(r-1)$ edges. Therefore, by Lemma \ref{lem:girth-cycle}, it has no cycle of length at most $e$. Thus, $f_r(n,(r-1)e,e) \le h_r(n,e)$.
\end{proof}

\section{Counting 3-graphs with girth greater than 5} \label{sec:pf-count-girth>5}

Here we prove \cref{thm:count-girth>5}, 
which says that 
	for every fixed $r \ge 3$, 
	the number of $r$-graphs on $n$ vertices with girth greater than $5$ 
	is $2^{o(n^{3/2})}$.

\begin{proof}[Proof of \cref{thm:count-girth>5}]
Let $F_r(n)$ denote the number of $r$-graphs on $n$ labeled vertices with girth greater than $5$. 

First we show that, for every $r \ge 3$, 
	one has $F_r(n) \le F_3(n)^{\binom{r}{3}}$, 
	thereby reducing the problem to $r=3$. 
Indeed, given an $r$-graph $H$, color the triples contained in each edge of $H$ arbitrarily from $1$ to $\binom{r}{3}$, using one color for each triple.
The color classes give a list $H_1, \dots, H_{\binom{r}{3}}$ of 3-graphs, each with girth greater than $5$, and thus there are at most $F_3(n)^{\binom{r}{3}}$ possibilities for such a list.
Furthermore, one can recover $H$ from the list $H_1, \dots, H_{\binom{r}{3}}$ since two triples in $H_1 \cup \cdots \cup H_{\binom{r}{3}}$ share two vertices if and only if they are contained in the same edge in $H$ (recall that $H$ has no 2-cycles). Thus, there are at most $F_3(n)^{\binom{r}{3}}$ possibilities for $H$.

Now it remains to show that $F_3(n) = 2^{o(n^{3/2})}$.
Let $H$ be a 3-graph with girth greater than 5.
By \cref{cor:girth>5}, $H$ has $o(n^{3/2})$ edges.
Let $G$ be the underlying shadow graph
	(a pair of vertices form an edge of $G$ if they are contained in a triple of $H$).
Since $H$ has girth greater than $5$, 
	every edge in $G$ lies in a unique triangle, 
	and $G$ is $C_4$-free with $o(n^{3/2})$ edges.
Then, by \cref{prop:count-c4-free-o} below, there are $2^{o(n^{3/2})}$ possibilities for $G$. The 3-graph $H$ can be recovered uniquely from $G$ and thus the number of such $H$ is also $2^{o(n^{3/2})}$.
\end{proof}

\begin{proposition} \label{prop:count-c4-free-o}
The number of $C_4$-free graphs on $n$ vertices with $o(n^{3/2})$ edges is $2^{o(n^{3/2})}$.	
\end{proposition}

\cref{prop:count-c4-free-o} can be proved by modifying the proof of the following classic result of Kleitman and Winston~\cite{KW82}.

\begin{theorem}[Kleitman--Winston] \label{thm:kw}
The number of $C_4$-free graphs on $n$ vertices is $2^{O(n^{3/2})}$. 
\end{theorem}

We follow the exposition of Samotij~\cite[Theorem 8]{Sam15} in his survey on counting independent sets in graphs via graph containers. We begin with the following key lemma.

\begin{lemma}[{\cite[Lemma 1]{Sam15}}] \label{lem:indep}
	Let $G$ be an $n$-vertex graph. Suppose that an integer $q$ and reals $R$ and $\beta \in [0,1]$ satisfy $R \ge e^{-\beta q} n$. Suppose that every subset $U \subset V(G)$ with $\abs{U} \ge R$ induces at least $\beta\binom{\abs{U}}{2}$ edges in $G$. Then, for every integer $m \ge q$, the number of $m$-element independent sets in $G$ is at most $\binom{n}{q}\binom{R}{m-q}$.
\end{lemma}

As in \cite{Sam15}, let $g_n(d)$ denote the maximum number of ways to attach a vertex of degree $d$ to an $n$-vertex $C_4$-free graph with minimum degree at least $d-1$ in such a way that the resulting graph remains $C_4$-free. In \cite{Sam15}, it was proved that $\max_{d \le n} g_n(d) \le  e^{O(\sqrt{n})}$. We modify the proof of this statement to obtain the following bound.

\begin{lemma} \label{lem:gnd}
	If $i \le n$ and $d = o(\sqrt{n})$, then $g_i(d) = e^{o(\sqrt{n})}$.
\end{lemma}

\begin{proof}
	If $d \le \sqrt{n}/(\log n)^2$, then $g_i(d) \le \binom{i}{d} \le n^d = e^{o(\sqrt{n})}$. So assume $d > \sqrt{n}/(\log n)^2$.
	
	Let $G$ be an $i$-vertex $C_4$-free graph with minimum degree at least $d-1$. Let $H$ be the square of $G$, i.e., $V(H) = V(G)$
	and two vertices are adjacent in $H$ 
	if and only if they are connected by a path of two edges in $G$. 
	Note that attaching a new vertex to $G$ will not create any 4-cycles if and only if the neighborhood of the new vertex is an independent set in $H$. It remains to upper bound the number of $d$-element independent sets in $H$ using \cref{lem:indep}.
	
	Let $R = 2n/(d-1)$, $\beta = (d-1)^2/2n$, and $q = \lceil 3 (\log n)^5 \rceil$. 
	Since $d > \sqrt{n}/(\log n)^2$, for $n$ sufficiently large we have $\beta q \ge \log n$ and thus $e^{-\beta q} i \le e^{-\beta q} n \le 1 \le R$.
	
	Since $G$ has minimum degree at least $d-1$, every $B \subseteq V(H)$ satisfies 
		$\sum_{z \in V(G)} \deg_G(z,B) \ge (d-1)\abs{B}$. 
	Thus, if $\abs{B} \ge R = 2n/(d-1)$, then the number of edges that $B$ induces in $H$ is equal to
	\begin{align*}
		\sum_{z \in V(G)} \binom{\deg_G(z,B)}{2} 
		&\ge i \binom{\sum_{z \in V(G)} \deg_G(z,B)/i}{2} 
		\\
		&\ge i \cdot \frac{(d-1)\abs{B}}{2i}\paren{\frac{(d-1)\abs{B}}{i} - 1}
		\ge \frac{(d-1)^2}{2n}\binom{\abs{B}}{2} = \beta \binom{\abs{B}}{2},
	\end{align*}
	where the first inequality uses the convexity of $x \mapsto \binom{x}{2} = \frac{x(x-1)}{2}$. 
	
	Applying \cref{lem:indep}, the number of $d$-element independent sets in $H$ is at most 
	\[
	\binom{i}{q}\binom{R}{d-q}
	\le n^q \paren{\frac{eR}{d-q}}^{d-q}
	\le e^{O((\log n)^6)} \paren{\frac{ 2ne}{(d-q)^2}}^{d-q}.
	\]
	Applying $\lim_{x \to 0} \sqrt{x} \log  (1/x)= 0$ with $x = (d-q)^2/(2ne) = o(1)$, we see that the right-hand side above is $e^{o(\sqrt{n})}$. Thus, the number of $d$-element independent sets in $H$ is $e^{o(\sqrt{n})}$.
\end{proof}

\begin{lemma}\label{lem:c4-free-min-deg}
	A $C_4$-free graph with minimum degree $d$ must have more than $d^3/2-d^2/2$ edges.
\end{lemma}

\begin{proof}
	Let $G$ be a $C_4$-free graph with minimum degree $d$ and $v$ a vertex of degree $d$.
	As $G$ is $C_4$-free, each vertex in $N(v)$ is adjacent to at most one vertex in $N(v)$. Thus, each vertex in $N(v)$ has at least 
	$d-2$ neighbors not in $\{v\} \cup N(v)$. As $G$ is $C_4$-free, each vertex not in $\{v\} \cup N(v)$ has at most one neighbor in $N(v)$, so $G$ has at least $1+d+d(d-2)=d^2-d+1$ vertices. 
	Hence, the number of edges in $G$ is at least $|V(G)|d/2 \geq (d^2-d+1)d/2 > d^3/2-d^2/2$. 
\end{proof}

\begin{proof}[Proof of \cref{prop:count-c4-free-o}]
By iteratively peeling off lowest-degree vertices, we see that every $n$-vertex graph has an ordering $v_1, \dots, v_n$ of vertices (in reverse order of peeling) such that, for each $i$, $v_i$ is a minimum-degree vertex in the subgraph $G_i$ induced by $\{v_1, \dots, v_i\}$. Letting $d_i$ denote the degree of $v_i$ in $G_i$, we see that the minimum degree of $G_{i-1}$ is at least $d_i - 1$. 

By \cref{lem:c4-free-min-deg}, every induced subgraph of a $C_4$-free graph with $m$ edges has minimum degree $O(m^{1/3})$. In particular, if the graph has $n$ vertices and $m=o(n^{3/2})$, then $d_i=o(\sqrt{n})$ for all $i$.

For each fixed ordering of the vertices ($n!$ possibilities) and each fixed sequence $d_2, \dots, d_n$ of degrees (at most $n!$ possibilities), noting that there are at most $g_i(d_{i+1})$ ways to attach $v_{i+1}$ to $G_i$, we see that the number of $C_4$-free graphs with these parameters is at most $g_1(d_2) g_2(d_3) \cdots g_{n-1}(d_n)$. By \cref{lem:gnd}, this count is $e^{o(n^{3/2})}$, even after summing over the at most $n!^2$ possibilities for the parameters.
\end{proof}

\end{document}